\documentclass[a4paper,11pt,reqno]{amsart}

\usepackage{dsfont,amssymb}

\usepackage{hyperref}
\usepackage[nobysame,alphabetic,initials,msc-links]{amsrefs}
\DeclareMathAlphabet{\mathpzc}{OT1}{pzc}{m}{it}

\DefineSimpleKey{bib}{how}
\renewcommand{\eprint}[1]{#1}
\BibSpec{misc}{%
  +{}{\PrintAuthors}  {author}
  +{,}{ \textit}      {title}
  +{.}{ }             {how}
  +{}{ \parenthesize} {date}
  +{,}{ available at \eprint}{eprint}
  +{,}{ available at \url}{url}
  +{,}{ }             {note}
  +{.}{}              {transition}
}

\mathchardef\mhyph="2D

\numberwithin{equation}{section}

\newtheorem{theorem}{Theorem}[section]

\newtheorem{lemma}[theorem]{Lemma}
\newtheorem{proposition}[theorem]{Proposition}
\theoremstyle{remark}

\newtheorem{example}[theorem]{Example}
\theoremstyle{definition}

\newcommand\bp{\begin{proof}}
\newcommand\ep{\end{proof}}
\newcommand\ee{\nopagebreak\mbox{\ }\hfill$\diamondsuit$}

\DeclareMathOperator{\End}{End}
\DeclareMathOperator{\Hilb}{Hilb}
\DeclareMathOperator{\Hom}{Hom}

\DeclareMathOperator{\Nat}{Nat}
\DeclareMathOperator{\Irr}{Irr}
\DeclareMathOperator{\Rep}{Rep}

\DeclareMathOperator\Tens{\mathpzc{Tens}}
\DeclareMathOperator\Tenssi{\mathpzc{Tens}_{si}}
\DeclareMathOperator{\Tr}{Tr}

\DeclareMathOperator\YD{\mathpzc{YD}_{brc}}

\newcommand\Dhat{\hat\Delta}
\newcommand{\C}{{\mathbb C}}

\newcommand{\B}{{\mathcal B}}
\newcommand\CC{{\mathcal C}}
\newcommand\D{\mathcal D}
\newcommand\E{\mathcal E}
\newcommand{\F}{{\mathcal F}}
\newcommand{\G}{{\mathcal G}}
\newcommand{\M}{{\mathcal M}}
\newcommand\PP{{\mathcal P}}
\newcommand\SSS{\mathcal S}
\newcommand\TT{\mathcal T}
\newcommand\U{\mathcal U}
\newcommand\X{\mathcal X}

\newcommand\un{{\mathds 1}}
\newcommand\trhd{\mathbin{\tilde{\rhd}}}
\newcommand{\alg}{\textrm{alg}}

\usepackage{wasysym}
\newcommand{\circt}%
{\mathbin{%
\mathchoice
{\ooalign{$\ocircle$\cr\hidewidth\raise-.15ex\hbox{$\scriptstyle\top\mkern2.05mu$}\cr}}
{\ooalign{$\ocircle$\cr\hidewidth\raise-.15ex\hbox{$\scriptstyle\top\mkern2.05mu$}\cr}}
{\ooalign{$\scriptstyle\ocircle$\cr\hidewidth\raise-.12ex\hbox{$\scriptscriptstyle\top\mkern1mu$}\cr}}
{\ooalign{$\scriptstyle\ocircle$\cr\hidewidth\raise-.12ex\hbox{$\scriptscriptstyle\top\mkern1mu$}\cr}}
}}

\begin{document}

\title[Duality for Yetter--Drinfeld algebras]{Categorical duality for Yetter--Drinfeld algebras}

\author[S. Neshveyev]{Sergey Neshveyev}

\email{sergeyn@math.uio.no}

\address{Department of Mathematics, University of Oslo,
P.O. Box 1053 Blindern, NO-0316 Oslo, Norway}

\thanks{The research leading to these results has received funding from the European Research Council
under the European Union's Seventh Framework Programme (FP/2007-2013) / ERC Grant Agreement no. 307663
}

\author[M. Yamashita]{Makoto Yamashita}

\email{yamashita.makoto@ocha.ac.jp}

\address{Department of Mathematics, Ochanomizu University,
Otsuka 2-1-1, 192-0361, Tokyo, Japan}

\thanks{Supported by the Danish National Research Foundation through the Centre
for Symmetry and Deformation (DNRF92), and by JSPS KAKENHI Grant Number 25800058}

\date{October 16, 2013; new version May 26, 2014; minor corrections October 23, 2014}

\begin{abstract}
We study tensor structures on $(\Rep G)$-module categories defined by actions of a compact quantum group $G$ on unital C$^*$-algebras. We show that having a tensor product which defines the module structure is equivalent to enriching the action of $G$ to the structure of a braided-commutative Yetter--Drinfeld algebra. This shows that the category of braided-commutative Yetter--Drinfeld $G$-C$^*$-algebras is equivalent to the category of generating unitary tensor functors from $\Rep G$ into C$^*$-tensor categories. To illustrate this equivalence, we discuss coideals of quotient type in $C(G)$, Hopf--Galois extensions and noncommutative Poisson boundaries.
\end{abstract}

\maketitle

\bigskip

\section*{Introduction}
This paper is a contribution to the study of compact quantum group actions from the categorical point of view. The idea of this approach can be traced to works of Wassermann~\cite{MR990110} and Landstad~\cite{MR1190512} in the 1980s. From the modern point of view, they proved that there is a one-to-one correspondence between full multiplicity ergodic actions of a compact group~$G$ and unitary fiber functors $\Rep G\to\Hilb_f$. The quantum analogue of this result in the purely algebraic setting was proved by Ulbrich~\cite{MR1006625} and Schauenburg~\cite{MR1408508}, and the corresponding result in the C$^*$-algebraic setting was proved by Bichon, De Rijdt and Vaes~\cite{MR2202309}. Thus, for a compact quantum group $G$, there is a correspondence between unitary fiber functors on $\Rep G$ and full quantum multiplicity ergodic actions of $G$. It is natural to ask then what corresponds to unitary tensor functors from $\Rep G$ into arbitrary C$^*$-tensor categories. We show that this is braided-commutative Yetter--Drinfeld algebras, which are algebras equipped with actions of $G$ and $\hat G$ satisfying certain compatibility conditions. Therefore such algebras play the same role for general tensor functors as Hopf--Galois objects for fiber functors, at least in the C$^*$-setting.

This can also be interpreted as follows. As has recently been shown in~\citelist{\cite{MR3121622} \cite{neshveyev-mjm-categorification}}, actions of~$G$ can be described in terms of $(\Rep G)$-module categories. Then our result says that such a module category structure is defined by a tensor functor if and only if we can also define an action of $\hat G$ to get a braided-commutative Yetter--Drinfeld algebra.

As an application, we strengthen a result of Tomatsu~\cite{MR2335776} characterizing coideals of quotient type in $C(G)$, and extend this characterization to Hopf--Galois objects. We also show that the correspondence between Yetter--Drinfeld algebras and tensor functors provides a rigorous link between Izumi's theory of Poisson boundaries of discrete quantum groups~\cite{MR1916370} and categorical Poisson boundaries we introduced in~\cite{NY-categorical-Poisson-boundary}.

\smallskip

\paragraph{\bf Acknowledgement} Part of this research was carried out
while the authors were attending the workshop ``Noncommutative
Geometry'' at  Mathematisches Forschungsinstitut Oberwolfach in
September 2013.  We thank the organizers and the staff for their
hospitality. M.Y.~thanks P.~Schauenburg and K.~Shimizu for bringing
his attention to~\cite{MR2863377} and~\cite{MR3039775}.

\bigskip

\section{Preliminaries} \label{sprelim}

In this section we briefly summarize the theory of compact quantum groups and their actions on operator algebras in the C$^*$-algebraic formulation, as well as discuss an algebraic approach to Yetter--Drinfeld C$^*$-algebras.

\smallskip

We mainly follow the conventions of~\cite{neshveyev-tuset-book}.  When $A$ and $B$ are C$^*$-algebras, $A \otimes B$ denotes their minimal tensor product. Unless said otherwise, we assume that C$^*$-categories are closed under subobjects. On the other hand, for C$^*$-tensor categories we do not assume that the unit object is simple. For objects $U$ and $V$ in a category $\CC$ we denote by $\CC(U,V)$ the set of morphisms $U\to V$.

\subsection{Compact quantum groups}
\label{sec:comp-quant-groups}

A \emph{compact quantum} group $G$ is represented by a unital C$^*$-algebra $C(G)$ equipped with a unital $*$-homomorphism $\Delta \colon C(G)\to C(G)\otimes C(G)$ satisfying the coassociativity $(\Delta \otimes \iota) \Delta = (\iota \otimes \Delta) \Delta$ and cancellation properties, meaning that $(C(G) \otimes 1) \Delta(C(G))$ and $(1 \otimes C(G)) \Delta(C(G))$ are dense in $C(G) \otimes C(G)$.  There is a unique state $h$ satisfying $(h \otimes \iota) \Delta = h$ (and/or $(\iota \otimes h) \Delta = h$) called the Haar state.  If $h$ is faithful, $G$ is called a reduced quantum group, and we are mainly interested in such cases.

A finite dimensional unitary representation of $G$ is a unitary element $U\in B(H_U)\otimes C(G)$, where~$H_U$ is a finite dimensional Hilbert space, such that $(\iota\otimes\Delta)(U)=U_{12}U_{13}$. The dense $*$-subalgebra of $C(G)$ spanned by matrix coefficients of finite dimensional representations is denoted by $\C[G]$.  The intertwiners between two representations $U$ and $V$ are the linear maps $T$ from $H_U$ to $H_V$ satisfying $V (T \otimes 1) = (T \otimes 1) U$.  The tensor product of two representations $U$ and $V$ is defined by $U_{13}V_{23}$ and denoted by $U\circt V$.  The category $\Rep G$ of finite dimensional unitary representations with intertwiners as morphisms and with tensor product $\circt$ becomes a semisimple C$^*$-tensor category.

Using the monoidal structure on $\Rep G$, for any $W \in \Rep G$, we can define an endofunctor $\iota \otimes W$ on $\Rep G$ which maps an object $U$ to $U \circt W$ and a morphism $T$ to $T\otimes\iota$.  A natural transformation between such functors $\iota \otimes W$ and $\iota \otimes V$ is given by a collection of morphisms $\eta_U\colon U \circt W \to U \circt V$ for $U \in \Rep G$ that are natural in $U$.

Denote the Woronowicz character $f_1\in\U(G)=\C[G]^*$ by~$\rho$. The space $\U(G)$ has the structure of a $*$-algebra, defined by duality from the Hopf $*$-algebra $(\C[G],\Delta)$. Every finite dimensional unitary representation $U$ of $G$ defines a $*$-representation $\pi_U$ of $\U(G)$ on $H_U$ by $\pi_U(\omega)=(\iota\otimes\omega)(U)$. We will often omit $\pi_U$ in expressions. Using the element $\rho$ the conjugate unitary representation to $U$ is defined by
$$
\bar U=(j(\rho)^{1/2}\otimes1)(j\otimes\iota)(U^*)(j(\rho)^{-1/2}\otimes1)\in B(\bar H_U)\otimes\C[G],
$$
where $j$ denotes the canonical $*$-anti-isomorphism $B(H_U)\cong B(\bar H_U)$ defined by $j(T)\bar\xi=\overline{T^*\xi}$. We have morphisms $R_U\colon \un\to \bar U\circt U$ and $\bar R_U\colon \un \to U\circt \bar U$ defined by
$$
R_U(1)=\sum_i\bar\xi_i\otimes\rho^{-1/2}\xi_i\ \ \text{and}\ \ \bar R_U(1)=\sum_i\rho^{1/2}\xi_i\otimes\bar\xi_i,
$$
where $\{\xi_i\}_i$ is an orthonormal basis in $H_U$. They solve the conjugate equations for $U$ and $\bar U$, meaning that
$$
(R^*_U\otimes\iota)(\iota\otimes \bar R_U)=\iota_{\bar U}\ \ \text{and}\ \ (\bar R_U^*\otimes\iota)(\iota\otimes R_U)=\iota_U.
$$

Therefore $\Rep G$ is a rigid C$^*$-tensor category. Woronowicz's Tannaka--Krein duality theorem recovers the $*$-Hopf algebra $\C[G]$ from the rigid semisimple C$^*$-tensor category $\Rep G$ and the forgetful fiber functor $U \mapsto H_U$.

\subsection{\texorpdfstring{$G$}{G}-algebras and \texorpdfstring{$(\Rep G)$}{(Rep G)}-module categories}
\label{sec:g-algebras}

Given a compact quantum group $G$, a unital $G$-C$^*$-algebra is a unital C$^*$-algebra $B$ equipped with a continuous left action $\alpha\colon B\to C(G)\otimes B$ of~$G$. This means that $\alpha$ is an injective unital $*$-homomorphism such that $(\Delta\otimes\iota)\alpha=(\iota\otimes\alpha)\alpha$ and such that the space $(C(G)\otimes1)\alpha(B)$ is dense in $C(G)\otimes B$.
The linear span of spectral subspaces,
$$
\B = \{ x \in B \mid \alpha(x) \in \C[G] \otimes_\alg B \},
$$
which is a dense $*$-subalgebra of $B$, is called the \textit{regular subalgebra} of $B$, and the elements of~$\B$ are called \emph{regular}.  More concretely, the algebra $\B$ is spanned by the elements of the form $(h\otimes\iota)((x\otimes1)\alpha(a))$ for $x\in\C[G]$ and $a\in B$. This algebra is of central importance for the categorical reconstruction of~$B$.

When $\D$ is a C$^*$-category, the category $\End(\D)$ of endofunctors of $\D$, with bounded natural transformations as morphisms, forms a C$^*$-tensor category.  A C$^*$-category $\D$ endowed with a unitary tensor functor from $\Rep G$ to the opposite of $\End(\D)$ is called a \emph{right $(\Rep G)$-module category}.  For $U \in \Rep G$, we denote the induced functor on $\D$ by $X \mapsto X \times U$.  An object $X$ in a $(\Rep G)$-module category $\D$ is said to be \emph{generating} if any other object $Y \in \D$ is isomorphic to a subobject of $X \times U$ for some $U \in \Rep G$.

Let us summarize the categorical duality theory of continuous actions of reduced compact quantum groups on unital C$^*$-algebras developed in~\cite{MR3121622} and~\cite{neshveyev-mjm-categorification}.

\begin{theorem}[\citelist{\cite{MR3121622}*{Theorem~6.4} \cite{neshveyev-mjm-categorification}*{Theorem~3.3}}]
\label{tactions}
Let $G$ be a reduced compact quantum group.  Then the following two categories are equivalent:
\begin{enumerate}
\item The category of unital $G$-C$^*$-algebras $B$ with unital $G$-equivariant $*$-homomorphisms as morphisms.
\item The category of pairs $(\D,M)$, where $\D$ is a right $(\Rep G)$-module C$^*$-category and $M$ is a generating object in $\D$, with equivalence classes of unitary $(\Rep G)$-module functors respecting the prescribed generating objects as morphisms.
\end{enumerate}
\end{theorem}

We omit the precise definition of the equivalence relation on functors between pairs $(\D,M)$, since it will not be important to us, see \cite{MR3121622}*{Theorem~7.1} for details. Note also that, as follows from the proof, under the above correspondence the fixed point algebra $B^G$ is isomorphic to~$\End_\D(M)$.

\smallskip

In the following subsections we overview the proof of the theorem.

\subsection{From algebras to module categories}
\label{sec:from-algebra-categ}

Given a $G$-C$^*$-algebra $(B, \alpha)$, we consider the category $\D_B$ of $G$-equivariant finitely generated right Hilbert $B$-modules.  In other words, objects of $\D_B$ are finitely generated right Hilbert $B$-modules $X$ equipped with a linear map $\delta=\delta_X\colon X\to C(G)\otimes X$ which satisfies the comultiplicativity property  $(\Delta\otimes\iota)\delta=(\iota\otimes\delta)\delta$, such that $(C(G)\otimes 1)\delta(X)$ is dense in $C(G)\otimes X$, and such that $\delta$ is compatible with the Hilbert $B$-module structure in the sense that
\begin{align*}
  \delta(\xi a) &= \delta(\xi)\alpha(a),&
  \langle\delta(\xi),\delta(\zeta)\rangle &= \alpha(\langle\xi,\zeta\rangle),
\end{align*}
 for $\xi, \zeta\in X$ and $a\in B$.  Here, $C(G)\otimes X$ is considered as a right Hilbert $(C(G)\otimes B)$-module.

For $X \in \D_B$ and $U \in \Rep G$, we obtain a new object $X \times U$ in $\D_B$ given by the linear space $H_U \otimes X$, which is a right Hilbert $B$-module such that
$$
(\xi \otimes x) a = \xi \otimes x a \quad\text{and}\quad \langle \xi \otimes x, \eta \otimes y \rangle_B = (\eta, \xi) \langle x, y \rangle_B \quad \text{for}\ \ \xi, \eta \in H_U,\ x, y \in X,\ a \in B,
$$
together with the compatible $C(G)$-coaction map
\begin{equation}\label{ecomod}
\delta=\delta_{H_U\otimes X} \colon H_U\otimes X \to C(G)\otimes H_U\otimes X, \quad \delta(\xi\otimes x)=U^*_{21}(\xi\otimes\delta_X(x))_{213}.
\end{equation}
This construction is natural both in $X$ and $U$, and satisfies $(X \times U) \times V \cong X \times (U \circt V)$, with the obvious isomorphism mapping $\zeta\otimes\xi\otimes x\in H_V\otimes H_U\otimes X$ into $\xi\otimes\zeta\otimes x$. We keep the notation $H_U \otimes X$ when we want to emphasize the realization of the object $X \times U$ as a Hilbert module.  This way $\D_B$ becomes a right $(\Rep G)$-module category.

It is known that by the stabilization argument, any object in $\D_B$ is a direct summand of $B \times U$ for some $U \in \Rep G$.  Thus we may, and often will, consider $\D_B$ as an idempotent completion of $\Rep G$ via the correspondence $U \mapsto B \times U$. To be precise, we start from a C$^*$-category with the same objects as in $\Rep G$, but with the new enlarged morphism sets
$$
\CC_B(U,V)=\Hom_{G,B}(H_U\otimes B,H_V\otimes B),
$$
and form new objects from projections in the C$^*$-algebras $\CC_B(U,U)$, thus obtaining a C$^*$-category $\CC_B$. Note that, more explicitly, the set $\CC_B(U,V)$ consists of elements $T\in B(H_U,H_V)\otimes B$ such that
$$
V^*_{12}(\iota\otimes\alpha)(T)U_{12}=T_{13}.
$$
Note also that we automatically have $\CC_B(U,V)\subset B(H_U,H_V)\otimes\B$.

For every $W \in \Rep G$, the functor $\iota\otimes W$ on $\Rep G$ extends to $\CC_B$ in the obvious way: given a morphism $T\in\CC_B(U,V)\subset B(H_U,H_V)\otimes\B$ the corresponding morphism $T\otimes\iota\in\CC_B(U \circt W, V \circt W)$ is $T_{13}\in B(H_U,H_V)\otimes B(H_W)\otimes \B$.  The right $(\Rep G)$-module C$^*$-categories $\CC_B$ and $\D_B$ are equivalent via the functor mapping $U$ into $B\times U$.

Although the category $\CC_B$ might appear somewhat ad hoc compared to $\D_B$, it is more convenient for computations and some of the constructions become simpler for $\CC_B$. For example, suppose that $f\colon B_0 \to B_1$ is a morphism of $G$-C$^*$-algebras.  Then, $\iota \otimes f$ defines linear transformations $f_{\#, U, V}\colon \CC_{B_0}(U, V) \to \CC_{B_1}(U, V)$, which together define a functor $f_\#\colon \CC_{B_0} \to \CC_{B_1}$.  The pair $(f_\#, \iota)_{U, V}$ gives a $(\Rep G)$-module homomorphism in the sense of~\cite{MR3121622}*{Definition~3.17}.  Under the above equivalence $\D_B \simeq \CC_B$, this obvious construction corresponds to the scalar extension functor $\D_{B_0} \to \D_{B_1}$, mapping $X$ into $X\otimes_{B_0}B_1$, discussed in~\cite{MR3121622}.  Note also that for the composition of $G$-equivariant maps we have the desired equality of functors $f_\# g_\# = (f g)_\#$ between the categories~$\CC_B$, rather than a natural isomorphism of functors, which we would have for the categories~$\D_B$.

\subsection{From module categories to algebras}
\label{sec:tann-krein-reconstr}

We recall the construction of an action from a pair~$(\D,M)$ following~\cite{neshveyev-mjm-categorification}. Without loss of generality we may assume that the $(\Rep G)$-module category $\D$ is strict. Furthermore, in order to simplify the notation, by replacing $\D$ by an equivalent category we may assume that it is the idempotent completion of the category $\Rep G$ with larger morphism sets $\D(U,V)$ than in $\Rep G$, such that~$M$ is the unit object $\un$ in $\Rep G$ and the functor $\iota\times U$ on $\D$ is an extension of the functor $\iota\otimes U$ on $\Rep G$. Namely, we simply define the new set of morphisms between $U$ and $V$ as $\D(M\times U,M\times V)$.

Choose representatives~$U_s$ of isomorphism classes of irreducible representations of $G$, and assume that $U_e=\un$ for some index~$e$. We write $H_s$ instead of $H_{U_s}$. Consider the linear space
\begin{equation}
  \label{eq:regular-subalg}
  \B=\bigoplus_s(\bar H_s\otimes \D(\un,U_s)).
\end{equation}
We may assume that $\Rep G$ is small and consider also the much larger linear space
\begin{equation}
  \label{eq:univ-regular-alg}
  \tilde\B=\bigoplus_U(\bar H_U\otimes \D(\un,U)),
\end{equation}
where the summation is over all objects in $\Rep G$. Define a linear map $\pi\colon\tilde\B\to\B$ as follows. Given a finite dimensional unitary representation $U$, choose isometries $w_i\colon H_{s_i}\to H_U$ defining a decomposition of $U$ into irreducibles. Then, for $\bar\xi\otimes T\in \bar H_U\otimes\D(\un,U)$, put
$$
\pi(\bar\xi\otimes T)=\sum_i\overline{w_i^*\xi}\otimes w_i^*T.
$$
This map is independent of any choices. The space $\tilde\B$ is an associative algebra with product
$$
(\bar\xi\otimes T)\cdot(\bar\zeta\otimes S)=\overline{(\xi\otimes\zeta)}\otimes (T\otimes\iota)S.
$$
This product defines a product on $\B$ such that $\pi(x)\pi(y)=\pi(x\cdot y)$ for all $x,y\in\tilde\B$.

In order to define the $*$-structure on $\B$, first define an antilinear map $\bullet$ on $\tilde\B$ by
\begin{equation}\label{estarmod}
(\bar\xi\otimes T)^\bullet=\overline{\overline{\rho^{-1/2}\xi}}\otimes( T^*\otimes\iota)\bar R_U\ \ \text{for}\ \ \bar\xi\otimes T\in\bar H_U\otimes \D(\un,U).
\end{equation}
This map does not define an involution on $\tilde\B$, but on $\B$ we get an involution such that $\pi(x)^*=\pi(x^\bullet) $ for all $x\in \tilde\B$.

The $*$-algebra $\B$ has a natural left $\C[G]$-comodule structure defined by the map $\alpha\colon\B\to\C[G]\otimes\B$ such that if $U$ is a finite dimensional unitary representation of $G$, $\{\xi_i\}_i$ is an orthonormal basis in~$H_U$ and $u_{ij}$ are the matrix coefficients of $U$ in this basis, then
\begin{equation} \label{eaction}
\alpha(\pi(\bar\xi_i\otimes T))=\sum_ju_{ij}\otimes\pi(\bar\xi_j\otimes T).
\end{equation}
It is shown then that the action $\alpha$ is algebraic in the sense of \cite{MR3121622}*{Definition~4.2}, meaning that the fixed point algebra $A=\B^G\cong\End_\D(\un)$ is a unital C$^*$-algebra and the conditional expectation $(h\otimes\iota)\alpha\colon\B\to A$ is positive and faithful. It follows that there is a unique completion of $\B$ to a C$^*$-algebra $B$ such that $\alpha$ extends to an action of the reduced form of $G$ on $B$. This finishes the construction of an action from a module category.

\begin{example}\label{exfunalg}
Consider the action of $G$ on itself by left translations, so we consider the coproduct map of $C(G)$ as an action of $G$ on $C(G)$. It corresponds  to the module category $\D=\Hilb_f$ of finite dimensional Hilbert spaces, with the distinguished object $\C$, considered as a $(\Rep G)$-module category using the forgetful tensor functor $U\mapsto H_U$. Explicitly, identifying $\D(\un,H_U)$ with $H_U$, we get an isomorphism of the algebra $\tilde\B$ constructed from the pair $(\Hilb_f,\C)$ onto $\widetilde{\C[G]}=\bigoplus_U\bar H_U\otimes H_U$, and then an isomorphism $\B\cong\C[G]$ such that $\pi\colon\tilde\B\to\B$ turns into the map $\pi_G\colon \widetilde{\C[G]}\to \C[G]$ that sends $\bar\xi\otimes\zeta\in\bar H_U\otimes H_U$ into the matrix coefficient $((\cdot\,\zeta,\xi)\otimes\iota)(U)$ of $U$.\ee
\end{example}

Returning to the general case, consider the action $\alpha\colon B\to C(G)\otimes B$ defined by a pair $(\D,M)$ as described above. The equivalence between the $(\Rep G)$-module categories $\D$ and $\D_B$ can be very concretely described as follows. First of all, as we have discussed, by replacing $\D$ by an equivalent category we may assume that it is the idempotent completion of $\Rep G$ with new morphisms sets. Similarly, instead of $\D_B$ we consider the category $\CC_B$. Then in order to define an equivalence we just have to describe the isomorphisms $\D(U,V)\cong\CC_B(U,V)$. The equivalence between $\D$ and $\CC_B$ constructed in the proof of \cite{neshveyev-mjm-categorification}*{Theorem~2.3} (see also Section~3 there) has the property that a morphism $T\in\D(\un,V)$ is mapped into
$$
\sum_j\zeta_j\otimes\pi(\bar\zeta_j\otimes T)\in\CC_B(\un,V)\subset B(\C,H_V)\otimes B,
$$
where $\{\zeta_j\}_j$ is an orthonormal basis in $H_V$ and we identify $B(\C,H_V)\otimes B$ with $H_V\otimes B$. Now assume that we have a morphism $T\in\D(U,V)$. We can write it as $(\iota\otimes R^*_U)(S\otimes\iota)$, with $S=(T\otimes\iota)\bar R_U\in\D(\un,V\otimes\bar U)$. Choose an orthonormal basis $\{\xi_i\}_i$ in $H_U$. Then the morphism $S\otimes\iota_U$ defines the element
$$
\sum_{i,j}\zeta_j\otimes\bar\xi_i\otimes1\otimes\pi\big(\overline{(\zeta_j\otimes\bar\xi_i})\otimes S\big)\in\CC_B(U,V\otimes\bar U\otimes U)\subset B(H_U,H_V\otimes \bar H_U\otimes H_U)\otimes B,
$$
where we identify $B(H_U,H_V\otimes\bar H_U\otimes H_U)$ with $H_V\otimes\bar H_U\otimes B(H_U)$. It follows that $T=(\iota\otimes R^*_U)(S\otimes\iota)$ is mapped into
$$
\sum_{ij} R^*_U(\xi_i\otimes\,\cdot\,)\zeta_j\otimes\pi\big(\overline{(\zeta_j\otimes\bar\xi_i)}\otimes S\big)\in\CC_B(U,V)\subset B(H_U,H_V)\otimes B.
$$
Since $R^*_U(\xi_i\otimes\xi)=(\rho^{-1/2}\xi,\xi_i)$, we conclude that the isomorphism $\D(U,V)\cong\CC_B(U,V)$ is such that
$$
\D(U,V)\ni T\mapsto\sum_{i,j}\theta_{\zeta_j,\xi_i}\pi_U(\rho^{-1/2})\otimes\pi\big(\overline{(\zeta_j\otimes\bar\xi_i)}\otimes (T\otimes\iota)\bar R_U\big)
\in\CC_B(U,V),
$$
where $\theta_{\zeta_j,\xi_i}\in B(H_U,H_V)$ is the operator defined by $\theta_{\zeta_j,\xi_i}\xi=(\xi,\xi_i)\zeta_j$. This can also be written as
\begin{equation} \label{eequiv}
\D(U,V)\ni T\mapsto\sum_{i,j}\theta_{\zeta_j,\xi_i}\otimes\pi\big(\overline{(\zeta_j\otimes\overline{\rho^{-1/2}\xi_i})}\otimes (T\otimes\iota)\bar R_U\big)\in\CC_B(U,V)\subset B(H_U,H_V)\otimes B.
\end{equation}

\subsection{Yetter--Drinfeld algebras} \label{sHilbert}

Assume we have a continuous left action of $\alpha\colon B\to C(G)\otimes B$  of  a compact quantum group $G$ on a unital C$^*$-algebra $B$, as well as a continuous right action $\beta\colon B\to \M(B\otimes c_0(\hat G))$ of the dual discrete quantum group~$\hat G$. The action $\beta$ defines a left $\C[G]$-module algebra structure $\rhd\colon \C[G]\otimes B\to B$ on~$B$ by
$$
x\rhd a=(\iota\otimes x)\beta(x)\ \ \text{for} \ \ x\in\C[G]\ \ \text{and}\ \ a\in B.
$$
Here we view $c_0(\hat G)$ as a subalgebra of $\U(G)=\C[G]^*$. This structure is compatible with involution, in the sense that
\begin{equation}\label{estar}
x\rhd a^*=(S(x)^*\rhd a)^*.
\end{equation}
We say that $B$ is a {\em Yetter--Drinfeld $G$-C$^*$-algebra} if the following identity holds for all $x\in\C[G]$ and $a\in\B$:
\begin{equation} \label{eYD1}
\alpha(x\rhd a) =x_{(1)} a_{(1)}S(x_{(3)})\otimes ( x_{(2)}\rhd a_{(2)}),
\end{equation}
where we use Sweedler's sumless notation, so we write $\Delta(x)=x_{(1)}\otimes x_{(2)}$ and $\alpha(a)=a_{(1)}\otimes a_{(2)}$.
Note that the above identity implies that $\B\subset B$ is a submodule over $\C[G]$.

Yetter--Drinfeld $G$-C$^*$-algebras can be regarded as $D(G)$-C$^*$-algebras for the Drinfeld double $D(G)$ of $G$, and they are studied in the more general setting of locally compact quantum groups by Nest and Voigt~\cite{MR2566309}. It is not difficult to see that our definition is equivalent to theirs,\footnote{It should also  be taken into account that the definition of coproduct on  $C_0(\hat G)$ used in the theory of locally compact quantum groups is opposite to the one usually used for compact quantum groups.} but
the case of compact quantum groups allows for the above familiar
algebraic formulation, which is more convenient for our purposes. In
the case of reduced compact quantum groups we can make it purely
algebraic by getting rid of the right action $\beta$ altogether.

\begin{proposition} \label{palgYD}
Assume that $G$ is a reduced compact quantum group and $\alpha\colon B\to C(G)\otimes B$ is a continuous action of $G$ on a unital C$^*$-algebra $B$. Let $\B\subset B$ be the subalgebra of regular elements. Suppose that $\B$ is also a left $\C[G]$-module algebra such that conditions \eqref{estar} and \eqref{eYD1} are satisfied for all $x\in\C[G]$ and $a\in\B$. Then there exists a unique continuous right action $\beta\colon B\to \M(B\otimes c_0(\hat G))$ such that $x\rhd a=(\iota\otimes x)\beta(a)$ for all $x\in\C[G]$ and $a\in\B$.
\end{proposition}

\bp Let us show first that for any finite dimensional unitary representation $U=\sum_{i,j}m_{ij}\otimes u_{ij}$ of~$G$, where $m_{ij}$ are matrix units in $B(H_U)$, there exists a unital $*$-homomorphism $\beta_U\colon B\to B\otimes B(H_U)$ such that
$$
\beta_U(a)=\sum_{i,j}(u_{ij}\rhd a)\otimes m_{ij}\ \ \text{for all}\ \ a\in\B.
$$
From the assumption that $\B$ is a $\C[G]$-module algebra we immediately get that $\beta_U\colon\B\to \B\otimes B(H_U)$ is a unital homomorphism. Condition \eqref{estar} implies that this homomorphism is $*$-preserving. Thus, all we have to do is to show that $\beta_U$ extends to a $*$-homomorphism $B\to B\otimes B(H_U)$. For this observe that the Yetter--Drinfeld condition \eqref{eYD1} implies that
$$
(\alpha\otimes\iota)\beta_U(a)=U_{31}(\iota\otimes\beta_U)\alpha(a)U^*_{31}.
$$
It follows that if we let $B_U$ to be the norm closure of $\beta_U(\B)$ in $B\otimes B(H_U)$, then the restriction of the map
$$
B\otimes B(H_U)\ni y\mapsto U_{31}^*(\alpha\otimes\iota)(y)U_{31}\in C(G)\otimes B\otimes B(H_U)
$$
to $B_U$ gives us a well-defined unital $*$-homomorphism $\gamma\colon B_U\to C(G)\otimes B_U$. Furthermore, since $\gamma(\beta_U(a))=(\iota\otimes\beta_U)\alpha(a)$ for $a\in\B$, the map $\gamma$ defines a continuous action of $G$ on $B_U$. It follows that if we define a new C$^*$-norm $\|\cdot\|'$ on $\B$ by
$$
\|a\|'=\max\{\|a\|,\|\beta_U(a)\|\},
$$
then the action $\alpha$ of $G$ on $\B$ extends to a continuous action on the completion of $\B$ in this norm. But according to \cite{MR3121622}*{Proposition~4.4} a C$^*$-norm with such property is unique. Hence $\|a\|'=\|a\|$ for all $a\in\B$, and therefore the map $\beta_U$ extends by continuity to $B$.

\smallskip

Since $c_0(\hat G)\cong c_0\mhyph\bigoplus_s B(H_s)$, the homomorphisms $\beta_{U_s}$ define a unital $*$-homomorphism $\beta\colon B\to \M(B\otimes c_0(\hat G))=\ell^\infty\mhyph\bigoplus_s (B\otimes B(H_s))$ such that $(\iota\otimes x)\beta(a)=x\rhd a$  for all $x\in\C[G]$ and $a\in\B$. It is then straightforward to check that $\beta$ is a continuous action. The uniqueness is also clear.
\ep

\bigskip

\section{Yetter--Drinfeld algebras and tensor functors}
\label{sec:equiv-modul-over}

In this section we prove our main result, a categorical description of a class of Yetter--Drinfeld C$^*$-algebras.

\subsection{Two categories}\label{stwocategories} A Yetter--Drinfeld $G$-C$^*$-algebra $B$ is said to be {\em braided-commutative} if for all $a,b\in\B$ we have
\begin{equation}\label{eBC}
ab=b_{(2)}(S^{-1}(b_{(1)})\rhd a).
\end{equation}
When $b$ is in the fixed point algebra $A=\B^G$, the right hand side reduces to $b a$, and we see that $A$ is contained in the center of $\B$.

\smallskip

The following theorem is our principal result. A closely related result in the purely algebraic framework has been obtained by Brugui\`{e}res and Natale~\cite{MR2863377}.

\begin{theorem} \label{tcatch}
Let $G$ be a reduced compact quantum group.  Then the following two categories are equivalent:
\begin{enumerate}
\item\label{tcatch-item-alg}
The category $\YD(G)$ of unital braided-commutative Yetter--Drinfeld $G$-C$^*$-algebras with unital $G$- and $\hat G$-equivariant $*$-homomorphisms as morphisms.
\item\label{tcatch-item-cat}
The category $\Tens(\Rep G)$ of pairs $(\CC,\E)$, where $\CC$ is a C$^*$-tensor category and $\E\colon\Rep G\to\CC$ is a unitary tensor functor such that $\CC$ is generated by the image of $\E$. The set of morphisms  $(\CC,\E)\to(\CC',\E')$  in this category is the set of equivalence classes of pairs  $(\F,\eta)$, where $\F$ is a unitary tensor functor $\F\colon\CC\to\CC'$ and $\eta$ is a natural unitary monoidal isomorphism $\eta \colon \F\E \to \E'$.
\end{enumerate}
Moreover, given a morphism $[(\F,\eta)]\colon (\CC,\E)\to(\CC',\E')$, the corresponding homomorphism of Yetter--Drinfeld C$^*$-algebras is injective if and only if $\F$ is faithful, and it is surjective if and only if $\F$ is~full.
\end{theorem}

The condition that $\CC$ is generated by the image of $\E$ means that
any object in $\CC$ is isomorphic to a subobject of $\E(U)$ for some
$U\in\Rep G$.  We remind the reader that we assume that C$^*$-categories are
closed under subobjects. We also stress that we do not assume that the
unit in $\CC$ is simple. In fact, as will be clear from the proof, the
C$^*$-algebra $\End_\CC(\un)$ is exactly the fixed point algebra $B^G$
in the C$^*$-algebra $B$ corresponding to~$(\CC,\E)$.

We have to explain how we define the equivalence relation on pairs $(\F,\eta)$. Assume $(\F,\eta)$ is a pair consisting of a unitary tensor functor $\F\colon\CC\to\CC'$ and a natural unitary monoidal isomorphism $\eta \colon \F\E \to \E'$. Then, for all objects $U$ and $V$ in $\Rep G$, we get linear maps
$$
\CC(\E(U),\E(V))\to\CC'(\E'(U),\E'(V)),\ \ T\mapsto \eta_V\F(T)\eta_U^{-1}.
$$
We say that two pairs $(\F,\eta)$ and $(\tilde\F,\tilde\eta)$ are equivalent, if the corresponding maps $\CC(\E(U),\E(V))\to\CC'(\E'(U),\E'(V))$ are equal for all $U$ and $V$.

\smallskip

A somewhat more concrete way of thinking of the category $\Tens(\Rep G)$ of pairs $(\CC,\E)$ is as follows. Assume~$(\CC,\E)$ is such a pair. First of all observe that the functor $\E$ is automatically faithful by semisimplicity and existence of conjugates in $\Rep G$. Then replacing the pair $(\CC,\E)$ by an isomorphic one, we may assume that $\CC$ is a strict C$^*$-tensor category containing $\Rep G$ and $\E$ is simply the embedding functor. Namely, similarly to our discussion in Section~\ref{sec:tann-krein-reconstr}, define new sets of morphisms between objects $U$ and $V$ in $\Rep G$ as $\CC(\E(U),\E(V))$ and then complete the new category we thus obtain with respect to subobjects.

Assume now that we have two strict C$^*$-tensor categories $\CC$ and $\CC'$ containing $\Rep G$, and consider the embedding functors $\E\colon\Rep G\to \CC$ and $\E'\colon\Rep G\to\CC'$. Assume $[(\F,\eta)]\colon (\CC,\E)\to(\CC',\E')$ is a morphism. This means that the unitary isomorphisms $\eta_U\colon \F(U)\to U$ in $\CC'$ are such that $\F(T)=\eta_V^{-1}T\eta_U$ for any morphism $T\colon U\to V$ in $\Rep G$, and the morphisms $$\F_{2;U,V}\colon \F(U)\otimes\F(V)\to\F(U\otimes V)$$ defining the tensor structure of $\F$ are given by $\F_{2;U,V}=\eta_{U\circt V}^{-1}(\eta_U\otimes\eta_V)$. We can then define a new unitary tensor functor $\tilde\F$ from the full subcategory of $\CC$ formed by the objects in $\Rep G\subset\CC$ into $\CC'$ by letting $\tilde\F(U)=U$, $\tilde\F(T)=\eta_V\F(T)\eta_U^{-1}$ for $T\in\CC(U,V)$, and $\tilde\F_{2;U,V}=\iota$.  This functor can be extended to $\CC$, by sending any subobject $X \subset U$ with corresponding projection $p_X \in \End_\CC(U)$ to an object corresponding to the projection $\tilde\F(p_X)\in \End_{\CC'}(U)$. Such an extension is unique up to a natural unitary monoidal isomorphism. Then by definition $[(\F,\eta)]=[(\tilde\F,\iota)]$.

Therefore morphisms $(\CC,\E)\to(\CC',\E')$ are equivalence classes of unitary tensor functors $\F\colon\CC\to\CC'$ such that $\F$ is the identity functor on $\Rep G\subset\CC$ and $\F_{2;U,V}=\iota$ for all objects $U$ and $V$ in $\Rep G$. Two such functors~$\F$ and~$\G$ are equivalent, or in other words they define the same morphism, if $\F(T)=\G(T)$ for all morphisms $T\in\CC(U,V)$ and all objects $U$ and $V$ in $\Rep G$.

\smallskip

The rest of this section is devoted to the proof of Theorem~\ref{tcatch}.

\subsection{From Yetter--Drinfeld algebras to tensor categories}
\label{sec:from-yetter-drinfeld}

In this subsection the assumption that $G$ is reduced will not be important.

Assume that $B$ is a braided-commutative Yetter--Drinfeld
$G$-C$^*$-algebra. Consider the category~$\D_B$ of $G$-equivariant
finitely generated right Hilbert $B$-modules discussed in
Section~\ref{sec:from-algebra-categ}. Then $\D_B$ can be turned into
a C$^*$-tensor category. This construction is known for the $D(G)$-equivariant $B$-modules in the purely algebraic approach~\citelist{\cite{MR1289088}\cite{MR1291020}}, and our key observation is that the same formula works for the $G$-equivariant modules, see also~\cite{MR3039775}*{Section~3.7}. Let us say that, for $X \in \D_B$, a vector $\xi \in X$ is regular if $\delta_X(\xi)$ lies in the algebraic tensor product $\C[G]\otimes_\alg X$.

\begin{lemma} \label{lkey}
Assume that $X$ is a $G$-equivariant finitely generated right Hilbert $B$-module, and $\X$ be its subspace of regular vectors.  Then there exists a unique unital $*$-homomorphism $\pi_X\colon B\to\End_B(X)$ such that $\pi_X(a)\xi=\xi_{(2)}(S^{-1}(\xi_{(1)})\rhd a)$ for all $a\in \B$ and $\xi\in\X$. Furthermore, we have $\delta_X(\pi_X(a)\xi)=(\iota\otimes\pi_X)\alpha(a)\delta_X(\xi)$ for all $a\in B$ and $\xi\in X$.
\end{lemma}

\bp It suffices to consider the case $X=H_U\otimes B$ for an irreducible unitary representation $U=U_s$ of $G$, since any other module embeds into a finite direct sum of such modules as a direct summand. Then, using the action $\beta\colon B\to \M(B\otimes c_0(\hat G))$ and the projection $B\otimes c_0(\hat G)\to B\otimes B(H_s)\cong B(H_s)\otimes B$, we get a unital $*$-homomorphism $\pi_X\colon B\to \End_B(H_U\otimes B)=B(H_U)\otimes B$ such that
$$
\pi_X(a)=\sum_{i,j} m_{ij}\otimes (u_{ij}\rhd a),
$$
where $U=\sum_{i,j}m_{ij}\otimes u_{ij}$ and $m_{ij}$ are the matrix units in $B(H_U)$ defined by an orthonormal basis~$\{\xi_i\}_i$ in $H_U$. In order to see that this gives the correct definition of $\pi_X$, take $b\in\B$. Recalling definition~\eqref{ecomod} of $\delta_{H_U\otimes B}$, we get
$$
(\xi_i\otimes b)_{(1)}\otimes(\xi_i\otimes b)_{(2)}=\sum_{j}u_{ij}^*b_{(1)}\otimes(\xi_j\otimes b_{(2)}).
$$
Hence
$$
(\xi_i\otimes b)_{(2)}(S^{-1}((\xi_i\otimes b)_{(1)})\rhd a)=\sum_j\xi_j\otimes b_{(2)}(S^{-1}(u^*_{ij}b_{(1)})\rhd a)
=\sum_j\xi_j\otimes (S^{-1}(u^*_{ij})\rhd a)b,
$$
where the last equality follows by braided commutativity. Since $S^{-1}(u^*_{ij})=u_{ji}$, we see that $\pi_X(a)$ acts as stated in the formulation of the lemma.

\smallskip

In order to show that $\delta_X(\pi_X(a)\xi)=(\iota\otimes\pi_X)\alpha(a)\delta_X(\xi)$  we take an arbitrary $X$. It suffices to consider $a\in\B$. Then for $\xi\in\X$ we have
$$
\delta_X(\pi_X(a)\xi)=\delta_X(\xi_{(2)}(S^{-1}(\xi_{(1)})\rhd a))=\xi_{(2)}(S^{-1}(\xi_{(1)})\rhd a)_{(1)}\otimes \xi_{(3)}
(S^{-1}(\xi_{(1)})\rhd a)_{(2)}.
$$
Applying the Yetter--Drinfeld condition~\eqref{eYD1} we see that the last expression equals
$$
\xi_{(4)}S^{-1}(\xi_{(3)})a_{(1)}\xi_{(1)}\otimes \xi_{(5)}
(S^{-1}(\xi_{(2)})\rhd a_{(2)})=a_{(1)}\xi_{(1)}\otimes\xi_{(3)}(S^{-1}(\xi_{(2)})\rhd a_{(2)}),
$$
and this is exactly $(\iota\otimes\pi_X)\alpha(a)\delta_X(\xi)$.
\ep

If $X$ is as in the lemma, we conclude that $X$ has the structure of an $G$-equivariant $B$-$B$-correspon\-den\-ce.  If $f$ is a $G$-equivariant endomorphism of the right Hilbert $B$-module $X$, it is automatically a $B$-bimodule map because of the way the left action of $B$ is defined.  Therefore the category $\D_B$ can be considered as a full subcategory of the C$^*$-category of $G$-equivariant $B$-$B$-correspondences. The latter category has a natural C$^*$-tensor structure. In order to show that $\D_B$ forms a C$^*$-tensor subcategory it suffices to show that, given objects $X$ and $Y$ in $\D_B$, we have:
\begin{enumerate}
\item $X\otimes_B Y$ is a finitely generated right $B$-module;
\item the left $B$-module structure on $X\otimes_B Y$ induced by that on $X$ coincides with the left $B$-module structure given by Lemma~\ref{lkey} using the action of $G$ and the right $B$-module structure on $X\otimes_B Y$.
\end{enumerate}
The second property is a routine computation similar to the one in the proof of the second part of Lemma~\ref{lkey}, so we omit it. In order to check (i) it suffices to consider modules of the form $H_U\otimes B$. For such modules we have the following more precise result.

\begin{lemma}\label{liso}
For any finite dimensional unitary representations $U$ and $V$ of $G$, the map
$$
T_{U,V}\colon (H_V\otimes  B)\otimes_ B(H_U\otimes B)\to H_{U \circt V}\otimes B,\ \ (\zeta\otimes b)\otimes(\xi\otimes a)\mapsto \xi_{(2)}\otimes\zeta\otimes (S^{-1}(\xi_{(1)})\rhd b)a,
$$
is a $G$-equivariant unitary isomorphism of right Hilbert $B$-modules. Furthermore, the isomorphisms~$T_{U,V}$ have the property $T_{U\circt V,W}(\iota\otimes T_{U,V})=T_{U,V\circt W}(T_{V,W}\otimes\iota)$.
\end{lemma}

Recall that the $\C[G]$-comodule structure on $H_U$ is given by $\xi\mapsto \xi_{(1)}\otimes \xi_{(2)}=U^*_{21}(1\otimes\xi)$.

\bp[Proof of Lemma~\ref{liso}] Since $\xi_{(2)}\otimes (S^{-1}(\xi_{(1)})\rhd b)=b(\xi\otimes 1)$, it is clear that the map~$T_{U,V}$ defines a right $\B$-module isomorphism
$$
(H_V\otimes \B)\otimes_\B(H_U\otimes\B)\cong H_{U \circt V}\otimes \B.
$$
It is also obvious that $T_{U,V}$ is isometric on the subspace spanned by vectors of the form $(\zeta\otimes 1)\otimes(\xi\otimes 1)$. Since such vectors generate $(H_V\otimes \B)\otimes_\B(H_U\otimes\B)$ as a right $\B$-module, and this module is dense in $(H_V\otimes B)\otimes_B(H_U\otimes B)$, it follows that $T_{U,V}$ extends by continuity to a unitary isomorphism of right Hilbert $B$-modules.

\smallskip

Next let us check the $G$-equivariance. The $\C[G]$-comodule structure on $(H_V\otimes \B)\otimes_\B(H_U\otimes\B)$ is given~by
\begin{align*}
\delta((\zeta\otimes b)\otimes (\xi\otimes a))&=(\zeta\otimes b)_{(1)}(\xi\otimes a)_{(1)}\otimes(\zeta\otimes b)_{(2)}\otimes (\xi\otimes a)_{(2)}\\
&=\zeta_{(1)}b_{(1)}\xi_{(1)}a_{(1)}\otimes(\zeta_{(2)}\otimes b_{(2)})\otimes (\xi_{(2)}\otimes a_{(2)}).
\end{align*}
Applying $\iota\otimes T_{U,V}$ we get
$$
\zeta_{(1)}b_{(1)}\xi_{(1)}a_{(1)}\otimes\xi_{(3)}\otimes\zeta_{(2)}\otimes (S^{-1}(\xi_{(2)})\rhd b_{(2)})a_{(2)}.
$$
On the other hand, using the same symbol $\delta$ for the comodule structure on $H_{U \circt V}\otimes\B$, since $(U \circt V)^*=V^*_{23}U^*_{13}$ we get

\smallskip
$\displaystyle
(\iota\otimes\delta)T_{U,V}((\zeta\otimes b)\otimes (\xi\otimes a))
$
\begin{align*}
&=\delta(\xi_{(2)}\otimes\zeta\otimes(S^{-1}(\xi_{(1)})\rhd b)a)\\
&=\zeta_{(1)}\xi_{(2)}(S^{-1}(\xi_{(1)})\rhd b)_{(1)}a_{(1)}\otimes\xi_{(3)}\otimes\zeta_{(2)}\otimes (S^{-1}(\xi_{(1)})\rhd b)_{(2)}a_{(2)}.
\end{align*}
Applying \eqref{eYD1} we see that the last expression equals

\smallskip
$\displaystyle
\zeta_{(1)}\xi_{(2)}S^{-1}(\xi_{(1)})_{(1)}b_{(1)}S(S^{-1}(\xi_{(1)})_{(3)})a_{(1)} \otimes\xi_{(3)}\otimes\zeta_{(2)}\otimes (S^{-1}(\xi_{(2)})\rhd b_{(2)})a_{(2)}
$
\begin{align*}
&=\zeta_{(1)}\xi_{(4)}S^{-1}(\xi_{(3)})b_{(1)}\xi_{(1)}a_{(1)} \otimes\xi_{(5)}\otimes\zeta_{(2)}\otimes (S^{-1}(\xi_{(2)})\rhd b_{(2)})a_{(2)}\\
&=\zeta_{(1)}b_{(1)}\xi_{(1)}a_{(1)}\otimes\xi_{(3)}\otimes\zeta_{(2)}\otimes (S^{-1}(\xi_{(2)})\rhd b_{(2)})a_{(2)}.
\end{align*}
Therefore the map $T_{U,V}$ is indeed $G$-equivariant.

\smallskip

Finally, in order to prove the equality $T_{U \circt V,W}(\iota\otimes T_{U,V})=T_{U,V\circt W}(T_{V,W}\otimes\iota)$ it suffices to check it on tensor products of vectors of the form $\xi\otimes1$, since such tensor products generate a dense subspace of triple tensor products as right $\B$-modules. But for such vectors the statement is obvious.
\ep

Therefore the category $\D_B$ can be considered as a full C$^*$-tensor subcategory of the category of $G$-equivariant $B$-$B$-correspondences. In view of the previous lemma, it is convenient to replace the tensor product by the opposite one, so we put $X\times Y=Y\otimes_B X$.  Furthermore, the functor $\E_B\colon\Rep G\to\CC_B$ mapping~$U$ into the module $H_U\otimes B$, together with the unitary isomorphisms $T_{U,V}\colon \E_B(U)\otimes\E_B(V)\to \E_B(U\circt V)$ from Lemma~\ref{liso}, is a unitary tensor functor.  We have thus proved the following result.

\begin{theorem} \label{thm:brd-comm-YD-tensor-cat}
 Let $G$ be a compact quantum group and $B$ be a unital braided-commutative Yetter--Drinfeld $G$-C$^*$-algebra.  Then the $G$-equivariant finitely generated  right Hilbert $B$-modules form a C$^*$-tensor category $\D_B$ with tensor product $X \times Y=Y\otimes_BX$.  Furthermore, there is a unitary tensor functor $\E_B\colon\Rep G\to\D_B$ mapping~$U$ to the module $H_U\otimes B$.
\end{theorem}

Up to  an isomorphism, the pair $(\D_B,\E_B)$ can be more concretely described as follows. As we discussed in Section~\ref{sec:from-algebra-categ}, the category $\D_B$ is equivalent to the category $\CC_B$, which is the idempotent completion of the category with the same objects as in $\Rep G$, but with the new morphism sets
$$
\CC_B(U,V)\subset B(H_U,H_V)\otimes \B
$$
consisting of elements $T$ such that
$
V^*_{12}(\iota\otimes\alpha)(T)U_{12}=T_{13}.
$
We define the tensor product of objects in $\CC_B$ as in $\Rep G$,
and in order to completely describe the tensor structure it remains to write down a formula for the linear maps
$$
\CC_B(U,V)\otimes\CC_B(W,Z)\to\CC_B(U\circt W,V\circt Z).
$$
This can be done using Lemma~\ref{liso}. First, note that by the proof of Lemma~\ref{lkey}, if $U=\sum_{i,j}m_{ij}\otimes u_{ij}$ then for any $\xi\in H_U$ and $b\in\B$ we have
$$
\xi_{(2)}\otimes (S^{-1}(\xi_{(1)})\rhd b)=\sum_{ij}(m_{ij}\otimes (u_{ij}\rhd b))(\xi \otimes1).
$$
Therefore the map $T_{U,V}$ from Lemma~\ref{lkey} can be written as
$$
(\zeta\otimes b)\otimes(\xi\otimes a)\mapsto\sum_{i,j} m_{ij}\xi\otimes\zeta\otimes (u_{ij}\rhd b)a.
$$
It follows that given $T=\sum_l T_l\otimes b_l\in\CC_B(W,Z)$, the morphism $\iota\otimes T\in\CC_B(U\circt W,U\circt Z)$ considered as a map $H_{U\circt W}\otimes B\to H_{U \circt Z}\otimes B$ acts by
$$
\xi\otimes\zeta\otimes 1\mapsto T_{U,Z}\left(\sum_l (T_l\zeta\otimes b_l)\otimes(\xi\otimes 1)\right)=\sum_{i,j,l}m_{ij}\xi\otimes T_l\zeta\otimes (u_{ij}\rhd b_l).
$$
On the other hand, if $S=\sum_k S_k\otimes a_k\in\CC_B(U,V)$, then the morphism $S\otimes\iota\in\CC_B(U\circt Z,V\circt Z)$ considered as a map $H_{U\circt W}\otimes B\to H_{U\circt Z}\otimes B$ acts by
$$
\xi\otimes\zeta\otimes 1\mapsto T_{V,Z}\left(\sum_k (\zeta\otimes 1)\otimes (S_k\xi\otimes a_k)\right)
=\sum_k S_k\xi\otimes\zeta\otimes a_k.
$$
To summarize, the tensor structure on $\CC_B$ is described by the following rules:
\begin{align}
&\text{if}\ \ T=\sum_l T_l\otimes b_l\in\CC_B(W,Z),\ \ \text{then}\ \ \iota_U\otimes T=\sum_{i,j,l}m_{ij}\otimes T_l\otimes (u_{ij}\rhd b_l);\label{etensor1}\\
&\text{if}\ \ S\in\CC_B(U,V),\ \ \text{then}\ \ S\otimes\iota_Z=S_{13}.\label{etensor2}
\end{align}
In this picture the functor $\E_B\colon\Rep G\to\D_B$ becomes the strict tensor functor $\F_B\colon\Rep G\to\CC_B$ which is the identity map on objects, while on morphisms it is $T\mapsto T\otimes1$.

\subsection{From tensor categories to Yetter--Drinfeld algebras}
\label{sec:from-tens-categ}

Let us turn to the construction of a Yetter--Drinfeld algebra from a pair~$(\CC,\E)\in\Tens(\Rep (G))$. The category~$\CC$ can be considered as a right $(\Rep G)$-module category with the distinguished object~$\un$. Therefore by Theorem~\ref{tactions} we can construct a C$^*$-algebra $B = B_\CC$ together with a left continuous action $\alpha\colon B\to C(G)\otimes B$. Our goal is to prove the following.

\begin{theorem} \label{tfromcat}
The $G$-C$^*$-algebra $B$ corresponding to the $(\Rep G)$-module category $\CC$ with the distinguished object~$\un$ has a natural structure of a braided-commutative Yetter--Drinfeld C$^*$-algebra.
\end{theorem}

The construction of the Yetter--Drinfeld structure can be described for any pair $(\CC,\E)$, but in order to simplify the notation we assume that $\CC$ is strict, $\Rep G$ is a C$^*$-tensor subcategory of $\CC$ and~$\E$~is simply the embedding functor. This is enough by the discussion following the formulation of Theorem~\ref{tcatch}.

Recall from Section~\ref{sprelim} that the subalgebra $\B\subset B$ of regular elements is given by \eqref{eq:regular-subalg}. By Proposition~\ref{palgYD}, to prove the theorem we have to define a $\C[G]$-module algebra structure on $\B$ satisfying properties~\eqref{estar}, \eqref{eYD1}, and \eqref{eBC}.

In Section~\ref{sprelim} we also defined a `universal' algebra $\tilde\B=\bigoplus_U(\bar H_U\otimes\CC(\un,U))$, together with a homomorphism $\pi\colon\tilde\B\to\B$. Recall from Example~\ref{exfunalg} that we denote by $\widetilde{\C[G]}=\bigoplus_U(\bar H_U\otimes H_U)$ the algebra $\tilde\B$ corresponding to the forgetful fiber functor $\Rep G\to\Hilb_f$, and then the corresponding homomorphism $\pi_G\colon\widetilde{\C[G]}\to\C[G]$ maps $\bar\xi\otimes\zeta\in\bar H_U\otimes H_U$ into $((\cdot\,\zeta,\xi)\otimes\iota)(U)$.

Define a linear map
$$
\trhd\colon \widetilde{\C[G]}\otimes\tilde\B\to\tilde\B
$$
by letting, for $\bar\xi\otimes\zeta\in \bar H_U\otimes H_U$ and $\bar\eta\otimes T\in \bar H_V\otimes \CC(\un,V)$,
\begin{equation}
\label{eq:from-tensor-to-dual-action}
(\bar\xi\otimes\zeta)\trhd(\bar\eta\otimes T)=\overline{(\xi\otimes\eta\otimes\overline{\rho^{-1/2}\zeta})}\otimes(\iota\otimes T\otimes\iota)\bar R_U\in \bar H_{U \circt V\circt \bar U}\otimes \CC(\un,U\circt V\circt\bar U).
\end{equation}
We remind that $\bar R_U\colon\un\to U\circt\bar U$ is given by
$\bar R_U(1)=\sum_i\rho^{1/2}\xi_i\otimes\bar\xi_i$
for an orthonormal basis~$\{\xi_i\}$ in~$H_U$. Identifying $\C[G]$ with the subspace $\bigoplus_s (\bar H_s\otimes H_s)\subset\widetilde{\C[G]}$, we define a linear map $$\rhd\colon\C[G]\otimes\B\to \B\ \ \text{by letting}\ \ x\rhd a=\pi(x\trhd a)\ \ \text{for}\ \ x\in\C[G]\ \ \text{and}\ \ a\in\B.$$

\begin{lemma}\label{lrhd}
The map $\rhd$ defines a left $\C[G]$-module algebra structure on $\B$, and we have
$$\pi_G(x)\rhd \pi(a)=\pi(x\trhd a)\ \ \text{for all}\ \ x\in\widetilde{\C[G]}\ \ \text{and}\ \ a\in\tilde\B.$$
\end{lemma}

\bp We start with the second statement.  We have to show that $\pi_G(x)\rhd\pi(a)=\pi(x\trhd a)$ for $x\in\widetilde{\C[G]}$ and $a\in\tilde\B$. Take $x=\bar\xi\otimes\zeta\in\bar H_U\otimes H_U$ and $a=\bar\eta\otimes T\in\bar H_V\otimes \CC(\un,V)$. Choose isometries $u_i\colon H_{s_i}\to H_U$ and $v_j\colon H_{s_j}\to H_V$ defining decompositions of $U$ and $V$ into irreducibles. Then
\begin{align}
\pi_G(x)\rhd\pi(a)&=\pi\left(\sum_{i,j}(\overline{u^*_i\xi}\otimes u^*_i\xi)\trhd(\overline{v^*_j\eta}\otimes v^*_jT)\right)\nonumber\\
&=\pi\left(\sum_{i,j}\overline{(u^*_i\xi\otimes v^*_j\eta\otimes\overline{\rho^{-1/2}u^*_i\zeta})}\otimes (\iota\otimes v^*_jT\otimes\iota)\bar R_{s_i}\right),\label{emodule1}
\end{align}
where $\bar R_{s_i}=\bar R_{U_{s_i}}$. On the other hand,
\begin{align}
\pi(x\trhd a)&=\pi\Big(\overline{(\xi\otimes\eta\otimes\overline{\rho^{-1/2}\zeta})}\otimes(\iota\otimes T\otimes\iota)\bar R_U\Big)\nonumber\\
&=\pi\left(\sum_{i,j,k}\overline{(u^*_i\xi\otimes v^*_j\eta\otimes\bar u^*_k\overline{\rho^{-1/2}\zeta})}\otimes (u^*_i\otimes v^*_jT\otimes\bar u^*_k)\bar R_{U}\right),\label{emodule2}
\end{align}
where the morphism $\bar u_k\colon H_{\bar U_{s_k}}=\bar H_{s_k}\to H_{\bar U}=\bar H_U$ is defined by $\bar u_k\bar\xi=\overline{u_k\xi}$. Since $u^*_k\pi_U(\rho)=\pi_{U_{s_k}}(\rho) u^*_k$, $\bar R_U=\sum_i (u_i\otimes\bar u_i)\bar R_{s_i}$
and the partial isometries $u_i$ have mutually orthogonal images, we see that expressions~\eqref{emodule1} and~\eqref{emodule2} are equal.

\smallskip
In order to show that $\rhd$ defines a left $\C[G]$-module structure, take $x=\bar\xi\otimes\zeta\in\bar H_U\otimes H_U$, $y=\bar\mu\otimes\nu\in\bar H_W\otimes H_W$ and $a=\bar\eta\otimes T\in\bar H_V\otimes \CC(\un,V)$. Then
$$
x\trhd(y\trhd a)=\overline{(\xi\otimes\mu\otimes\eta\otimes{\overline{\rho^{-1/2}\nu}}\otimes \overline{\rho^{-1/2}\zeta})}\otimes (\iota\otimes\iota\otimes T\otimes\iota\otimes\iota)(\iota\otimes \bar R_W\otimes\iota)\bar R_U
$$
is an element in $\bar H_{U\circt W\circt V\circt\bar W\circt \bar U}\otimes\CC(\un,U\circt W\circt V\circt\bar W\circt\bar U)$, and
$$
(x\cdot y)\trhd a =\overline{(\xi\otimes\mu\otimes\eta\otimes\overline{(\rho^{-1/2}\zeta\otimes \rho^{-1/2}\nu)})}\otimes (\iota\otimes T\otimes\iota)\bar R_{U\circt W}
$$
is an element in $\bar H_{(U\circt W)\circt V\circt \overline{(U\circt W)}}\otimes \CC(\un,(U\circt W)\circt V\circt\overline{(U\circt W)})$. The only reason why these two elements are different is that the representations $\bar W\circt\bar U$ and $\overline{U\circt W}$ are equivalent, but not equal. The map $\sigma\colon \bar H_W\otimes\bar H_U\to\overline{H_U\otimes H_W}$, $\sigma(\bar\mu\otimes\bar\xi)=\overline{\xi\otimes\mu}$ defines such an equivalence, and we have $\bar R_{U\circt W}=(\iota\otimes\iota\otimes\sigma)(\iota\otimes \bar R_W\otimes\iota)\bar R_U$. It follows that upon projecting to $\B$ we get an honest equality
$$
\pi(x\trhd(y\trhd a))=\pi((x\cdot y)\trhd a ),
$$
that is, $\pi_G(x)\rhd(\pi_G(y)\rhd\pi(a))=(\pi_G(x)\pi_G(y))\rhd\pi(a)$.

\smallskip

It remains to show that $\rhd$ respects the algebra structure on $\B$, that is, $x\rhd (ab)=(x_{(1)}\rhd a)(x_{(2)}\rhd b)$ for $x\in\C[G]$ and $a,b\in\B$.

Take elements $a=\bar\eta\otimes T\in \bar H_V\otimes \CC(\un,V)$ and $b=\bar\zeta\otimes S\in \bar H_W\otimes \CC(\un,W)$ in $\tilde\B$. Let $U$ be a finite dimensional unitary representation of $G$. Choose an orthonormal bases~$\{\xi_i\}_i$ in~$H_U$ and denote by~$u_{ij}$ the corresponding matrix coefficients of $U$. Since $u_{ij}=\pi_G(\bar\xi_i\otimes\xi_j)$ and $\Delta(u_{ij})=\sum_ku_{ik}\otimes u_{kj}$, we then have to show that
$$
\pi\big((\bar\xi_i\otimes\xi_j)\tilde\rhd(a\cdot b)\big)=\sum_k\pi\big(((\bar\xi_i\otimes\xi_k)\tilde\rhd a)\cdot ((\bar\xi_k\otimes\xi_j)\tilde\rhd b)\big).
$$
We have
\begin{equation}\label{emodulealg}
(\bar\xi_i\otimes\xi_j)\tilde\rhd(a\cdot b)=
\overline{(\xi_i\otimes\eta\otimes\zeta\otimes\overline{\rho^{-1/2}\xi_j})}\otimes(\iota\otimes T\otimes S\otimes\iota)\bar R_U.
\end{equation}
On the other hand,

\smallskip
$
\displaystyle \sum_k((\bar\xi_i\otimes\xi_k)\tilde\rhd a)\cdot ((\bar\xi_k\otimes\xi_j)\tilde\rhd b)
$
\begin{align*}
&=\sum_k\Big(\overline{(\xi_i\otimes\eta\otimes\overline{\rho^{-1/2}\xi_k})}\otimes(\iota\otimes T\otimes\iota)\bar R_U)\Big)\cdot \Big(
\overline{(\xi_k\otimes\zeta\otimes\overline{\rho^{-1/2}\xi_j})}\otimes(\iota\otimes S\otimes\iota)\bar R_U\Big)\\
&=\sum_k\overline{(\xi_i\otimes\eta\otimes\overline{\rho^{-1/2}\xi_k}\otimes \xi_k\otimes\zeta\otimes\overline{\rho^{-1/2}\xi_j})}\otimes(\iota\otimes T\otimes\iota\otimes\iota\otimes  S\otimes \iota)(\bar R_U\otimes\bar R_U).
\end{align*}
Since $\sum_k\overline{\rho^{-1/2}\xi_k}\otimes \xi_k=R_U(1)$, and $R_U$ is, up to a scalar factor, an isomeric embedding of $\un$ into~$\bar U\circt U$, by applying $\pi$ to the above expression we get
$$
\pi\big(\overline{(\xi_i\otimes\eta\otimes\zeta\otimes\overline{\rho^{-1/2}\xi_j})}\otimes(\iota\otimes T\otimes R_U^*\otimes  S\otimes \iota)(\bar R_U\otimes\bar R_U)\big).
$$
Since $(R^*_U\otimes\iota)(\iota\otimes \bar R_U)=\iota$, this is exactly the expression we obtain by applying $\pi$ to~\eqref{emodulealg}.
\ep

We next check compatibility~\eqref{estar} of $\rhd$ with the $*$-structure.

\begin{lemma}
For all $x\in\C[G]$ and $a\in\B$ we have $x\rhd a^*=(S(x)^*\rhd a)^*$.
\end{lemma}

\bp Recall that the involution on $\B$ arises from the map $\bullet$ on $\tilde\B$ defined by~\eqref{estarmod}, so for $a=\bar\eta\otimes T\in\bar H_V\otimes\CC(\un,V)$ we have
$$
a^\bullet=\overline{\overline{\rho^{-1/2}\eta}}\otimes( T^*\otimes\iota)\bar R_V\in\bar H_{\bar V}\otimes\CC(\un,\bar V).
$$
Let us also define an antilinear map $\dagger$ on $\widetilde{\C[G]}$ by letting, for $x=\bar\xi\otimes\zeta\in\bar H_U\otimes H_U$,
$$
x^\dagger=\bar\zeta\otimes\xi.
$$
We then have $\pi_G(x^\dagger)=S(\pi_G(x))^*$. Indeed, using that $(\iota\otimes S)(U)=U^*$, we compute:
$$
S(\pi_G(x))^*=S(((\cdot\,\zeta,\xi)\otimes\iota)(U))^*
=((\cdot\,\zeta,\xi)\otimes\iota)(U^*))^*=((\cdot\,\xi,\zeta)\otimes\iota)(U)=\pi_G(x^\dagger).
$$

\smallskip

Turning now to the proof of the lemma, we have to show that
$$
\pi(x\trhd a^\bullet)=\pi((x^\dagger\trhd a)^\bullet).
$$
We compute:
\begin{align}
x\trhd a^\bullet&=(\bar\xi\otimes\zeta)\trhd\Big(\overline{\overline{\rho^{-1/2}\eta}}\otimes( T^*\otimes\iota)\bar R_V\Big)\nonumber\\
&=\overline{(\xi\otimes \overline{\rho^{-1/2}\eta}\otimes\overline{\rho^{-1/2}\zeta})}\otimes(\iota\otimes T^*\otimes\iota\otimes\iota)(\iota\otimes \bar R_V\otimes\iota)\bar R_U\label{estar1}
\end{align}
and
\begin{align}
(x^\dagger\trhd a)^\bullet &=((\bar\zeta\otimes\xi)\trhd a)^\bullet=
\Big(\overline{(\zeta\otimes\eta\otimes\overline{\rho^{-1/2}\xi})}\otimes (\iota\otimes T\otimes \iota)\bar R_U\Big)^\bullet\nonumber\\
&=\overline{\overline{(\rho^{-1/2}\zeta\otimes\rho^{-1/2}\eta\otimes\bar\xi)}}\otimes \big(\big((\iota\otimes T\otimes \iota)\bar R_U\big)^*\otimes\iota_{\overline{U \circt V\circt\bar U}}\big)\bar R_{U \circt V\circt \bar U},\label{estar2}
\end{align}
where we used that $\pi_{\bar U}(\rho)=j(\pi_U(\rho))^{-1}$, that is, $\rho\bar\xi=\overline{\rho^{-1}\xi}$. Similarly to the proof of the previous lemma, the main reason why expressions~\eqref{estar1} and~\eqref{estar2} are not equal is that the representations $U\circt \bar V\circt\bar U$ and $\overline{U \circt V\circt\bar U}$ are equivalent, but not equal. The map $\sigma(\xi\otimes\bar \eta\otimes\bar\zeta)=\overline{\zeta\otimes\eta\otimes\bar\xi}$ defines such an equivalence, and then
$$
\bar R_{U \circt V\circt \bar U}=(\iota\otimes\iota\otimes\iota\otimes\sigma)(\iota\otimes\iota \otimes R_U\otimes\iota\otimes\iota)(\iota\otimes \bar R_V\otimes \iota)\bar R_U.
$$
Since $(\bar R_U^*\otimes\iota)(\iota\otimes R_U)=\iota$, we get
$$
\big(\big((\iota\otimes T\otimes \iota)\bar R_U\big)^*\otimes\iota\big)\bar R_{U \circt V\circt \bar U}=\sigma(\iota\otimes T^*\otimes\iota\otimes\iota)(\iota\otimes \bar R_V\otimes\iota)\bar R_U.
$$
From this we see that upon applying $\pi$ expressions~\eqref{estar1} and~\eqref{estar2} indeed become equal.
\ep

Our next goal is to check the Yetter--Drinfeld condition~\eqref{eYD1}.

\begin{lemma} \label{lYD}
For all $x\in\C[G]$ and $a\in\B$ we have
$$
\alpha(x\rhd a) =x_{(1)} a_{(1)}S(x_{(3)})\otimes (x_{(2)}\rhd a_{(2)}).
$$
\end{lemma}

\bp Let $U$ and $V$ be finite dimensional unitary representations of $G$. Choose orthonormal bases~$\{\xi_i\}_i$ in~$H_U$ and $\{\eta_k\}_k$ in $H_V$, and let $u_{ij}$ and $v_{kl}$ be the matrix coefficients of $U$ and $V$, respectively. In order to simplify the computations assume that the vectors $\xi_i$ are eigenvectors of $\rho$, so $\rho\xi_i=\rho_i\xi_i$ for some positive number $\rho_i$.  Then the matrix coefficients of $\bar U$ in the basis $\{\bar\xi_i\}_i$ are given by
\begin{equation}\label{econjun}
\bar u_{ij}=\rho_i^{1/2}\rho_j^{-1/2}u_{ij}^*=\rho_i^{1/2}\rho_j^{-1/2}S(u_{ji}).
\end{equation}

Consider elements $x=u_{i_0j_0}\in\C[G]$ and $a=\pi(\eta_{k_0}\otimes T)\in\B$ for some $T\in\CC(\un,V)$. Recalling definition~\eqref{eaction} of the action~$\alpha$, we have
$$
\alpha(a)=\sum_k v_{k_0k}\otimes\pi(\bar\eta_k\otimes T).
$$
It follows that
\begin{multline}
x_{(1)} a_{(1)}S(x_{(3)})\otimes (x_{(2)}\rhd a_{(2)})
=\sum_{i,j,k}u_{i_0i}v_{k_0k}S(u_{jj_0})\otimes( u_{ij}\rhd\pi(\bar\eta_k\otimes T))\\
=\sum_{i,j,k}\rho^{-1/2}_ju_{i_0i}v_{k_0k}S(u_{jj_0})\otimes \pi\Big(\overline{(\xi_i\otimes\eta_k\otimes\bar\xi_j)} \otimes(\iota\otimes T\otimes \iota)\bar{R}_U\Big).\label{eYD3}
\end{multline}
On the other hand,
\begin{align}
\alpha(x\rhd a)&=\rho_{j_0}^{-1/2}\alpha\big(\pi\big(\overline{(\xi_{i_0}\otimes \eta_{k_0}\otimes\bar\xi_{j_0})}\otimes (\iota\otimes T\otimes\iota)\bar{R}_U\big)\big)\nonumber\\
&=\rho_{j_0}^{-1/2}\sum_{i,j,k}u_{i_0i}v_{k_0k}\bar u_{j_0j}\otimes\pi\big(\overline{(\xi_{i}\otimes \eta_{k}\otimes\bar\xi_{j})}\otimes (\iota\otimes T\otimes\iota)\bar R_U\big).\label{eYD4}
\end{align}
Since $\rho^{-1/2}_jS(u_{jj_0})=\rho_{j_0}^{-1/2}\bar u_{j_0j}$, we see that expressions~\eqref{eYD3} and~\eqref{eYD4} are equal.
\ep

It remains to check the braided commutativity condition~\eqref{eBC}.

\begin{lemma}
For all and $a,b\in\B$ we have $ab=b_{(2)}(S^{-1}(b_{(1)})\rhd a)$.
\end{lemma}

\bp Let $U$, $V$, $\{\xi_i\}_i$, $u_{ij}$, $\bar u_{ij}$ be as in the proof of the previous lemma. Note that by swapping the roles of $U$ and $\bar U$ in~\eqref{econjun} we get
$$
S^{-1}(u_{ij})=\rho^{-1/2}_i\rho^{1/2}_j\bar u_{ji}.
$$
(Recall again that $\rho\bar\xi_i=\overline{\rho^{-1}\xi_i}$.) Using this, take $P\in\CC(\un,U)$, $\bar\eta\otimes T\in\bar H_V\otimes\CC(\un,V)$ and for $a=\pi(\bar\eta\otimes T)$ and $b=\pi(\bar\xi_i\otimes P)$ compute:
\begin{align*}
b_{(2)}(S^{-1}(b_{(1)})\rhd a)&=\sum_j\pi(\bar\xi_j\otimes P)(S^{-1}(u_{ij})\rhd\pi(\bar\eta\otimes T))\\
&=\sum_j\rho^{-1/2}_i\rho^{1/2}_j\pi(\bar\xi_j\otimes P)\pi((\bar{\bar\xi}_j\otimes\bar\xi_i)\trhd(\bar\eta\otimes T))\\
&=\sum_j\rho^{-1/2}_i\rho^{1/2}_j\pi(\bar\xi_j\otimes P)\pi\Big(\overline{(\bar\xi_j\otimes\eta\otimes \overline{\rho^{-1/2}\bar\xi_i})}\otimes(\iota\otimes T\otimes\iota)\bar R_{\bar U}\Big)\\
&=\sum_j\rho^{1/2}_j\pi\Big(\overline{(\xi_j\otimes \bar\xi_j\otimes\eta\otimes\bar{\bar\xi}_i)}\otimes(P\otimes(\iota\otimes T\otimes\iota)\bar R_{\bar U})\Big).
\end{align*}
Denote by $w$ the map $\xi\mapsto\bar{\bar\xi}$ defining an equivalence between $U$ and $\bar{\bar U}$. Then $\bar R_{\bar U}=(\iota\otimes w)R_U$. Hence the above expression equals
\begin{multline*}
\sum_j\rho^{1/2}_j\pi\Big(\overline{(\xi_j\otimes \bar\xi_j\otimes\eta\otimes\xi_i)}\otimes(P\otimes(\iota\otimes T\otimes\iota)R_{U})\Big)\\
=\pi\Big(\overline{(\bar R_U(1)\otimes\eta\otimes\xi_i)}\otimes (\iota\otimes\iota\otimes T\otimes\iota) (\iota\otimes R_{U})P\Big).
\end{multline*}
Since $\bar R_U$ is, up to a scalar factor, an isometric embedding of $\un$ into $U\circt\bar U$, the last expression equals
$$
\pi\Big(\overline{(\eta\otimes\xi_i)}\otimes (\bar R^*_U\otimes T\otimes\iota) (\iota\otimes R_{U})P\Big)=\pi(\overline{(\eta\otimes\xi_i)}\otimes(T\otimes P)).
$$
But this is exactly $ab$.
\ep

This finishes the proof of Theorem~\ref{tfromcat}.

\subsection{Functoriality}
\label{sec:functoriality}

Consider the category $\YD(G)$ of unital braided-commutative Yetter--Drinfeld $G$-C$^*$-algebras. For every object $B$ we have constructed isomorphic pairs $(\D_B,\E_B)$ and~$(\CC_B,\F_B)$. Using the extension of scalars functor discussed at the end of Section~\ref{sec:from-algebra-categ}, either of this constructions extends to a functor, giving us two naturally isomorphic functors $\TT\colon\YD(G)\to\Tens(\Rep G)$ and $\tilde\TT\colon\YD(G)\to\Tens(\Rep G)$. Namely, giving a morphism $f\colon B_0\to B_1$ we have a functor $f_\#\colon\D_{B_0}\to\D_{B_1}$ which maps a $G$-equivariant finitely generated right Hilbert $B_0$-module $X$ into $X\otimes_{B_0}B_1$. We define a tensor structure on this functor by using the isomorphisms
$$
(X\otimes_{B_0}B_1)\otimes_{B_1}(Y\otimes_{B_0}B_1)\cong (X\otimes_{B_0}Y)\otimes_{B_0}B_1
$$
such that $(x\otimes a)\otimes (y\otimes b)\mapsto x\otimes y_{(2)}\otimes (S^{-1}(y_{(1)})\rhd a)b$. That these maps are indeed well-defined and that they give us a tensor structure on $f_\#$, is not difficult to check using arguments similar to those in the proof of Lemma~\ref{liso}. The tensor functor $f_\#$ together with the obvious isomorphisms $\eta_U\colon (H_U\otimes B_0)\otimes_{B_0}B_1\to H_U\otimes B_1$ define a morphism $(\D_{B_0},\E_{B_0})\to(\D_{B_1},\E_{B_1})$.

If we consider the map $B\mapsto (\CC_B,\F_B)$ instead of $B\mapsto(\D_B,\E_B)$, then the situation is even better: in this case the functor $f_\#\colon\CC_{B_0}\to\CC_{B_1}$ defined by a morphism $f\colon B_0\to B_1$ is a strict tensor functor, meaning that $f_\#(T\otimes S)=f_\#(T)\otimes f_\#(S)$ on morphisms. This follows immediately from equations~\eqref{etensor1} and~\eqref{etensor2} describing the tensor structure on the categories $\CC_B$.

\smallskip

Let us now construct a functor $\SSS$ in the opposite direction. It is possible to define this functor on the whole category $\Tens(\Rep G)$, but in order to simplify notation we will construct it only on the full subcategory $\Tenssi(\Rep G)$ consisting of pairs~$(\CC,\E)$ such that $\CC$ is a strict C$^*$-tensor category containing $\Rep G$, $\CC$ is generated by $\Rep G$, and that $\E$ is the embedding functor. Since the embedding functor $\Tenssi(\Rep G)\to\Tens(\Rep G)$ is an equivalence of categories, any functor $\Tenssi(\Rep G)\to \YD(G)$ extends to $\Tens(\Rep G)$, and this extension is unique up to a natural isomorphism.

Given two objects $(\CC_0,\E_0)$ and $(\CC_1,\E_1)$ in $\Tenssi(\Rep G)$, consider the corresponding Yetter--Drinfeld C$^*$-algebras~$B_0$ and~$B_1$, and take a morphism $[(\F,\eta)]\colon (\CC_0,\E_0)\to(\CC_1,\E_1)$. As we discussed after the formulation of Theorem~\ref{tcatch}, we may assume that the restriction of $\F$ to $\Rep G\subset\CC_0$ is the identity tensor functor and $\eta_U=\iota$. In this case it is obvious from the construction of the algebras $\B_i$ that the maps $\bar H_U\otimes\CC_0(\un,U)\ni \bar\xi\otimes T\mapsto\bar\xi\otimes\F(T)\in\bar H_U\otimes\CC_1(\un,U)$ define a unital $*$-homomorphism $\B_0\to\B_1$ that respects the $\C[G]$-comodule and $\C[G]$-module structures. It extends to a homomorphism $f$ of C$^*$-algebras by \cite{MR3121622}*{Proposition~4.5}. It is also clear by our definition of morphisms in the category of pairs $(\CC,\E)$ that $f$ depends only on the equivalence class of $(\F,\iota)$. We thus get a functor $\SSS\colon\Tenssi(\Rep G)\to\YD(G)$.

Furthermore, it is clear from the construction that the morphism $f\colon B_0\to B_1$ defined by a morphism $[(\F,\iota)]\colon (\CC_0,\E_0)\to(\CC_1,\E_1)$ is injective if and only if the maps $\CC_0(\un,U_s)\to\CC_1(\un,U_s)$, $T\mapsto\F(T)$, are injective for all $s$, and $f$ is surjective if and only if these maps are surjective. Using Frobenius reciprocity it is easy to see that the maps $\CC_0(\un,U_s)\to\CC_1(\un,U_s)$ are injective, resp. surjective, for all $s$ if and only if the maps $\CC_0(U,U_V)\to\CC_1(U,V)$ are injective, resp. surjective, for all objects $U$ and $V$ in $\Rep G\subset\CC_0,\CC_1$. Since the categories $\CC_i$ are generated by $\Rep G$, it follows that $f$ is injective if and only if $\F$ is faithful, and $f$ is surjective if and only if $\F$ is full.

It is also worth noting that since a morphism $T\in\CC(\un,U)$ is zero if and only if $T^*T=0$ in~$\End_\CC(\un)$, we have, given a morphism $[(\F,\eta)]\colon (\CC_0,\E_0)\to(\CC_1,\E_1)$, that $\F$ is faithful if and only if the homomorphism $\End_{\CC_0}(\un)\to\End_{\CC_1}(\un)$ is injective. On the C$^*$-algebra level this corresponds to the simple property that a morphism $B_0\to B_1$ of $G$-C$^*$-algebras for a reduced compact quantum group~$G$ is injective if and only if its restriction to the fixed point algebra $B^{G}_0$ is injective.

\subsection{Equivalence of categories}
\label{sec:end-proof-thm-tcatch}

To finish the proof of Theorem~\ref{tcatch} it remains to show that the functors $\TT\colon\YD(G)\to\Tens(\Rep G)$ or $\tilde\TT\colon\YD(G)\to\Tens(\Rep G)$, and $\SSS\colon\Tens(\Rep G)\to\YD(G)$ are inverse to each other up to an isomorphism.

\smallskip

Let us start with a strict C$^*$-tensor category $\CC$ containing $\Rep G$ and construct a braided-commu\-tative Yetter--Drinfeld C$^*$-algebra $B$ as described in Section~\ref{sec:from-tens-categ}. By Theorem~\ref{tactions} the  $(\Rep G)$-module C$^*$-categories $\CC$ and $\CC_B$ are equivalent. We will use the concrete form of this equivalence explained in Section~\ref{sec:tann-krein-reconstr}. Recall that $\CC_B$ is the idempotent completion of the category $\Rep G$ with morphisms $\CC_B(U,V)\subset B(H_U,H_V)\otimes\B$, and we have a unitary equivalence $\F\colon\CC\to\CC_B$ such that $\F(U)=U$ for $U\in\Rep G$, while the action of $\F$ on morphisms is given by \eqref{eequiv}, so
$$
\CC(U,V)\ni T\mapsto\sum_{i,j}\theta_{\zeta_j,\xi_i}\otimes \pi\big(\overline{(\zeta_j\otimes\overline{\rho^{-1/2}\xi_i})}\otimes (T\otimes\iota)\bar R_U\big),
$$
where $\{\xi_i\}_i$ and $\{\zeta_j\}_j$ are orthonormal bases in $H_U$ and $H_V$, respectively. We claim that $\F$ is a strict tensor functor on the full subcategory of $\CC$ consisting of objects $U\in\Rep G$. This tensor functor extends then to a unitary tensor functor on the whole category $\CC$. Thus, we have to show that $\F(S\otimes T)=\F(S)\otimes\F(T)$ on morphisms in $\CC$. Since $\F$ is an equivalence of right $(\Rep G)$-module categories, we already know that this is true for morphisms~$S$ in~$\CC$ and morphisms $T$ in $\Rep G$; this is also not difficult to check directly, since the formula for $\F(S)\otimes\iota$ does not involve the Yetter--Drinfeld structure, see~\eqref{etensor2}. Therefore it remains to check that $\F(\iota\otimes T)=\iota\otimes\F(T)$ for morphisms $T$ in $\CC$.

Take $T\in\CC(V,W)$. Let $\{\eta_k\}_k$ be an orthonormal basis in $H_W$. We then have
$$
\F(\iota_U\otimes T) =\sum_{i,j,k,l}\theta_{\xi_i\otimes\eta_k,\xi_j\otimes\zeta_l} \otimes\pi\Big(\overline{(\xi_i\otimes\eta_k\otimes \overline{(\rho^{-1/2}\xi_j\otimes\rho^{-1/2}\zeta_l)})}\otimes ((\iota\otimes T)\otimes\iota_{\overline{U\circt W}})\bar R_{U\circt W}\Big).
$$
Similarly to the proof of Lemma~\ref{lrhd}, using that $\overline{U\circt W}$ is equivalent to $\bar W \circt \bar U$ and that modulo this equivalence $\bar R_{U\circt W}$ coincides with $(\iota\otimes \bar R_W\otimes\iota)\bar R_U$, we see that the above expression equals
$$
\sum_{i,j,k,l}\theta_{\xi_i\otimes\eta_k,\xi_j\otimes\zeta_l} \otimes\pi\Big(\overline{(\xi_i\otimes\eta_k\otimes\overline{\rho^{-1/2}\zeta_l}\otimes \overline{\rho^{-1/2}\xi_j})}\otimes (\iota\otimes(T\otimes\iota)\bar R_W\otimes\iota)\bar R_U\Big).
$$
The operators $\theta_{\xi_i,\xi_j}$ are the matrix units $m_{ij}$ in $B(H_U)$. Recalling the definition of $\rhd$ we can therefore write the above expression as
$$
\sum_{i,j,k,l}m_{ij}\otimes\theta_{\eta_k,\zeta_l} \otimes \Big(u_{ij}\rhd\pi\Big(\overline{(\eta_k\otimes\overline{\rho^{-1/2}\zeta_l})}\otimes (T\otimes\iota)\bar R_W\Big)\Big),
$$
where $u_{ij}$ are the matrix units of $U$. According to~\eqref{etensor1} this is exactly the formula for $\iota_U\otimes\F(T)$.

\medskip

Conversely, consider a unital braided-commutative Yetter--Drinfeld C$^*$-algebra $B$ and the corresponding pair $(\CC_B,\F_B)$. Let $B_\CC$ be the Yetter--Drinfeld C$^*$-algebra constructed from this pair. By Theorem~\ref{tactions} we know that there exists an isomorphism $\lambda\colon B_\CC\to B$ intertwining the actions of $G$. So all we have to do is to check that $\lambda$ is also a $\C[G]$-module map. The isomorphism $\lambda$ is defined by
\begin{equation} \label{eiso}
\lambda(\pi(\bar\zeta\otimes T))=(\bar\zeta\otimes\iota)(T)
\end{equation}
for $\zeta\in H_V$ and $T\in\CC_B(\un,V)\subset B(\C,H_V)\otimes\B=H_V\otimes \B$, see the proof of \cite{neshveyev-mjm-categorification}*{Theorem~2.3}. As above, fix finite dimensional unitary representations $U$ and $V$ of $G$ and orthonormal bases~$\{\xi_i\}_i$ and~$\{\zeta_k\}_k$ in $H_U$ and $H_V$, and let $u_{ij}$ be the matrix coefficients of $U$. Take
$$
T=\sum_k\zeta_k\otimes b_k\in\CC_B(\un,V)\subset H_V\otimes\B.
$$
Then $\lambda(\pi(\bar\zeta_{k_0}\otimes T))=b_{k_0}$, and we want to check that
$$
\lambda(u_{i_0j_0}\rhd\pi(\bar\zeta_{k_0}\otimes T))=u_{i_0j_0}\rhd b_{k_0}.
$$
By definition we have
$$
u_{i_0j_0}\rhd\pi(\bar\zeta_{k_0}\otimes T)=\pi((\bar\xi_{i_0}\otimes\xi_{j_0})\trhd (\bar \zeta_{k_0}\otimes T))
=\pi\Big(\overline{(\xi_{i_0}\otimes\zeta_{k_0}\otimes\overline{\rho^{-1/2}\xi_{j_0}})}\otimes (\iota\otimes T\otimes\iota)\bar R_U\Big).
$$
In order to compute the image of this element under $\lambda$, we need an explicit formula for $(\iota\otimes T\otimes\iota)\bar R_U\in\CC_B(\un,U\otimes V\otimes\bar U)\subset H_{U \circt V \circt \bar U}\otimes\B$. By~\eqref{etensor1} and~\eqref{etensor2}, the element
$$
\iota_U\otimes T\otimes\iota_U \in\CC_B(U\circt \bar U,U\circt V\circt \bar U)\subset B(H_U)\otimes H_V\otimes B(\bar H_U)\otimes\B
$$
equals
$\sum_{i,j,k}m_{ij}\otimes\zeta_k\otimes 1\otimes( u_{ij}\rhd b_k).$ It follows that
$$
(\iota\otimes T\otimes\iota)\bar R_U=\sum_{i,j,k}(\xi_i\otimes\zeta_k\otimes\overline{\rho^{1/2}\xi_j})\otimes( u_{ij}\rhd b_k).
$$
Therefore
$$
u_{i_0j_0}\rhd\pi(\bar\zeta_{k_0}\otimes T)=\pi\left(\overline{(\xi_{i_0}\otimes\zeta_{k_0}\otimes\overline{\rho^{-1/2}\xi_{j_0}})}\otimes \left(\sum_{i,j,k}(\xi_i\otimes\zeta_k\otimes\overline{\rho^{1/2}\xi_j})\otimes( u_{ij}\rhd b_k)\right)\right).
$$
Applying $\lambda$ we get the required equality
$\lambda(u_{i_0j_0}\rhd\pi(\bar\zeta_{k_0}\otimes T))=u_{i_0j_0}\rhd b_{k_0}$. Since the algebra $\B$ is spanned by such elements $b_{k_0}$ for different $V$, it follows that $\lambda$ is a $\C[G]$-module map.  This completes the proof of Theorem~\ref{tcatch}.

\bigskip

\section{Coideals of quotient type and their generalizations}
\label{sec:examples}

In this section we illustrate Theorem~\ref{tcatch} by considering well-known examples of Yetter--Drinfeld algebras arising from quantum subgroups and Hopf--Galois extensions.

\subsection{Quotient type coideals}
\label{sec:quot-type-coid}

By a closed quantum subgroup of $G$ we mean a compact quantum group $H$ together with a surjective homomorphism $\pi\colon\C[G]\to\C[H]$ of Hopf $*$-algebras. This is consistent with the definition used in the theory of locally compact quantum groups, but is weaker than e.g.~the definition used in~\cite{MR2335776}. Assuming that both $G$ and $H$ are reduced, the homomorphism~$\pi$ does not always extend to a homomorphism $C(G)\to C(H)$. Nevertheless the algebra $C(G/H)$ of continuous functions on the quantum homogeneous space $G/H$ is always well-defined: it is the norm closure of
$$
\C[G/H]=\{x\in\C[G]\mid (\iota\otimes\pi)\Delta(x)=x\otimes1\}.
$$

The algebra $C(G/H)$ is a braided-commutative Yetter--Drinfeld $G$-C$^*$-algebra, with the left action of~$G$ defined by the restriction of $\Delta$ to $C(G/H)$, and the action of $\hat G$ defined by the restriction of the right adjoint action on $C(G)$ to $C(G/H)$. In other words, the $\C[G]$-module structure on $\C[G/H]$ is defined~by
$$
x\rhd a=x_{(1)}aS(x_{(2)}).
$$

It is known and is easy to see that the $G$-C$^*$-algebra $C(G/H)$ corresponds to the category $\Rep H$ with the distinguished object $\un$, viewed as a $(\Rep G)$-module category via the forgetful tensor functor $\Rep G\to\Rep H$. Namely, in the notation of Section~\ref{sec:from-tens-categ}, by identifying $\Hom_H(\C,H_U)$ with a subspace of $H_U$, we can view the algebra $\tilde\B$ corresponding to the functor $\Rep G\to\Rep H$ as a subalgebra of $\widetilde{\C[G]}=\bigoplus_U(\bar H_U\otimes H_U)$. Then the map $\pi_G\colon\widetilde{\C[G]}\to\C[G]$ induces a $G$-equivariant isomorphism $\B\cong\C[G/H]$.

We claim that the $\C[G]$-module structure on $\C[G/H]$ defined by the tensor functor $\Rep G\to\Rep H$ is exactly the adjoint action. In order to show this it is enough to consider the case of trivial~$H$, since it corresponds to the inclusion $\Rep G\hookrightarrow\Hilb_f$, while the general case corresponds to the intermediate inclusion $\Rep G\hookrightarrow\Rep H$. As in the proof of Lemma~\ref{lYD}, fix unitary representations~$U$ and~$V$ and orthonormal bases $\{\xi_i\}_i$ in $H_U$ and $\{\zeta_k\}_k$ in $H_V$ such that $\rho\xi_i=\rho_i\xi_i$. Denote matrix coefficients of $U$, $V$ and $\bar U$ by $u_{ij}$, $v_{kl}$, $\bar u_{ij}$. Recall that by~\eqref{econjun} we have $\bar u_{ij}=\rho_i^{1/2}\rho_j^{-1/2}S(u_{ji})$. Then
$$
(\bar\xi_i\otimes\xi_j)\trhd(\bar\zeta_k\otimes\zeta_l)
=\sum_m\overline{(\xi_i\otimes\zeta_k\otimes\rho_j^{-1/2}\bar\xi_j)}\otimes (\rho_m^{1/2}\xi_m\otimes\zeta_l\otimes\bar\xi_m)
$$
It follows that
$$
u_{ij}\rhd v_{kl}=\sum_m\rho_j^{-1/2}\rho_m^{1/2}u_{im}v_{kl}\bar u_{jm}=\sum_mu_{im}v_{kl}S(u_{mj}),
$$
which is exactly the formula for the adjoint action.

\smallskip

As a simple application of Theorem~\ref{tcatch} we now get the following result, which under slightly stronger assumptions has been already established in~\cite{MR2335776} and~\cite{MR2785890}.

\begin{theorem} \label{tTomSal}
Let $G$ be a reduced compact quantum group. Then any unital left $G$- and right $\hat G$-invariant C$^*$-subalgebra of $C(G)$ has the form $C(G/H)$ for a unique closed quantum subgroup~$H$ of~$G$.
\end{theorem}

\begin{proof}
Let $B\subset C(G)$ be a unital left $G$- and right $\hat G$-invariant C$^*$-algebra. Consider the corresponding pair $(\D_B,\E_B)=\TT(B)\in\Tens(\Rep G)$. By the ergodicity of the $G$-action on $B$, the unit object in $\D_B$ is simple. Since~$\D_B$ is generated by the image of $\Rep G$ and the category $\Rep G$ is rigid, the C$^*$-tensor category~$\D_B$ is rigid as well. The inclusion $B\hookrightarrow C(G)$ defines a morphism
$$
(\D_B,\E_B)\to(\D_{C(G)},\E_{C(G)})\cong(\Hilb_f,\F),
$$
where $\F\colon\Rep G\to\Hilb_f$ is the forgetful fiber functor. This means that $\D_B$ has a unitary fiber functor $\E\colon\D_B\to\Hilb_f$ such that $\F=\E\E_B$. By Woronowicz's Tannaka--Krein duality theorem, the pair $(\D_B,\E)$ defines a compact quantum group $H$. Then the functor $\E_B$ defines a functor $\Rep G\to\Rep H$ such that the forgetful fiber functor $\F$ on $\Rep G$ factors through that on $\Rep H$. It follows that $H$ can be regarded as a quantum subgroup of $G$.

Since by the discussion preceding the theorem the factorization of the fiber functor $\F\colon\Rep G\to\Hilb_f$ through $\Rep G\to\Rep H$ corresponds to the inclusion $C(G/H)\hookrightarrow C(G)$, we have therefore shown that there exists a closed quantum subgroup $H\subset G$ and an isomorphism $(\D_B,\E_B)\cong (\D_{C(G/H)},\E_{C(G/H)})$ such that the morphism $(\D_B,\E_B)\to(\D_{C(G)},\E_{C(G)})$ under this isomorphism becomes the morphism $(\D_{C(G/H)},\E_{C(G/H)})\to(\D_{C(G)},\E_{C(G)})$ defined by the inclusion $C(G/H)\hookrightarrow C(G)$. Since $\TT$ is an equivalence of categories, this implies that $B=C(G/H)$.

\smallskip

It remains to prove the uniqueness. In other words, we want to show that $C(G/H)\subset C(G)$ determines the kernel of the restriction map $\C[G]\to\C[H]$. Since $\C[G/H]$ is spanned by the matrix coefficients $a_{\xi,\zeta}=((\cdot\,\zeta,\xi)\otimes\iota)(U)$ such that $\zeta$ is an $H$-invariant vector, we can recover $\Hom_H(\un,U)\subset H_U$ for any representation $U$ of $G$ from $C(G/H)$. Using the duality morphisms in $\Rep G$ we can then recover $\Hom_H(V,U)\subset B(H_V,H_U)$ for all $V$ and $U$. Finally, observe that a finite combination $\sum_ia_{\xi_i,\zeta_i}$ of matrix coefficients in $\C[G]$, with $\xi_i,\zeta_i\in H_U$, is in the kernel of  the restriction map $\C[G]\to\C[H]$ if and only if $\sum_i(\cdot\,\zeta_i,\xi_i)$ vanishes on the commutant of $\End_H(H_U)$ in $B(H_U)$.
\end{proof}

\subsection{Invariant subalgebras of linking algebras}
\label{sec:linking-algebra}

The considerations of the previous subsection can be generalized to the linking algebras defined by monoidal equivalences.  Let $\F$ be the forgetful functor $\Rep G\to \Hilb_f$, and $\F'\colon \Rep G \to \Hilb_f, U \mapsto H'_U$, be another unitary fiber functor.  We denote the compact quantum group corresponding to $\F'$ by $G'$.  Then it is not difficult to check that the linking algebra between~$G$ and~$G'$, introduced in the C$^*$-algebraic setting in~\cite{MR2202309} and in the purely algebraic setting earlier in~\cite{MR1408508}, is exactly the C$^*$-algebra $B(\F,\F')$ corresponding to the pair $(\Hilb_f,\F')$ by our construction. In addition to the left action of $G$ it carries also a commuting right action of~$G'$, which is easy to see using that the regular subalgebra of $B(\F,\F')$ is $\B(\F,\F')=\bigoplus_s(\bar H_s\otimes H_s')$.

The $G$-C$^*$-algebras $B$ of the form $B(\F,\F')$ can be abstractly characterized by saying that the regular subalgebra $\B\subset B$ is a Hopf--Galois extension of $\C$ over $\C[G]$, which is a well-studied notion in the algebraic approach to quantum groups, see~\cite{arXiv:1006.3014}. By definition, this means that the Galois map
$$
\Gamma\colon \B \otimes \B \to \C[G] \otimes \B, \quad x \otimes y \mapsto x_{(1)} \otimes x_{(2)} y,
$$
is bijective. Analogously to the case of $\C[G]$ (which is the linking algebra $\B(\F,\F)$), there is a standard structure of a braided-commutative Yetter--Drinfeld algebra over $G$ on $\B$.  Namely, the action of $\C[G]$ on~$\B$ is the so called Miyashita--Ulbrich action, defined by $$x \rhd a = \Gamma^{-1}(x\otimes1)_1 a \Gamma^{-1}(x\otimes1)_2.$$

We claim that this action is the same as the one induced by the pair $(\Hilb_f,\F')$ by our construction. In order to show this, replace $(\Hilb_f,\F')$ by an isomorphic pair consisting of a strict C$^*$-tensor category $\CC$ containing $\Rep G$ and the embedding functor $\Rep G\to\CC$, as explained in Section~\ref{stwocategories}. What is now special about $\CC$, is that the unit object is simple and the maps $\CC(\un,U)\otimes\CC(\un,V)\to\CC(\un,U\circt V)$ are bijective. As in the previous subsection, fix unitary representations $U$ and $V$ of $G$ and an orthonormal basis $\{\xi_i\}_i$ in $H_U$ such that $\rho\xi_i=\rho_i\xi_i$. We can find elements $T_l\in\CC(\un, U)$ and $S_l\in\CC(\un,\bar U)$ such that
$$
\bar R_U=\sum_lT_l\otimes S_l\ \ \text{in}\ \ \CC.
$$
Then, for any $P\in\CC(\un, V)$ and $\zeta\in H_V$, we have
$$
(\bar\xi_i\otimes\xi_j)\tilde\rhd(\bar\zeta\otimes P)=\rho_j^{-1/2}\overline{(\xi_i\otimes\zeta\otimes\bar\xi_j)}\otimes(\iota\otimes P\otimes\iota)\bar R_U=\rho_j^{-1/2}\sum_l(\bar\xi_i\otimes T_l)\cdot(\bar\zeta\otimes P)\cdot(\bar{\bar\xi}_j\otimes S_l).
$$
Therefore in order to prove the claim it suffices to check that
$$
\sum_l\Gamma(\pi(\bar\xi_i\otimes T_l)\otimes\pi(\bar{\bar\xi}_j\otimes S_l))=\rho_j^{1/2}u_{ij}\otimes1.
$$
But this is true by the following simple computation:
\begin{multline*}
\sum_l\Gamma(\pi(\bar\xi_i\otimes T_l)\otimes\pi(\bar{\bar\xi}_j\otimes S_l))=
\sum_{k,l}u_{ik}\otimes\pi(\bar\xi_k\otimes T_l)\pi(\bar{\bar\xi}_j\otimes S_l)\\
=\sum_ku_{ik}\otimes\pi(\overline{(\xi_k\otimes\bar\xi_j)}\otimes \bar R_U)
=\sum_k u_{ik}\otimes \overline{\bar R^*_U(\xi_k\otimes\bar \xi_j)}
=\rho_j^{1/2}u_{ij}\otimes1.
\end{multline*}

\smallskip

If $H'$ is a closed quantum subgroup of $G'$, then, similarly to the C$^*$-algebras $C(G/H)\subset C(G)$, we may define C$^*$-algebras $B(\F,\F')^{H'}\subset B(\F,\F')$. Then by a completely analogous argument to that in the proof of Theorem~\ref{tTomSal} we obtain the following result.

\begin{theorem}
Let $G$ be a reduced compact quantum group and $B=B(\F,\F')$ be the linking C$^*$-algebra defined by the forgetful fiber functor $\F\colon\Rep G\to\Hilb_f$ and a unitary fiber functor $\F'\colon\Rep G\to\Hilb_f$. Let $G'$ be the compact quantum group defined by $\F'$. Then any unital left $G$- and right $\hat G$-invariant C$^*$-subalgebra of $B(\F,\F')$ has the form $B(\F,\F')^{H'}$ for a unique closed quantum subgroup $H'\subset G'$.
\end{theorem}

Let us finally say a few words about the differences between our approach to reconstructing the tensor functor $\F'$ from $B(\F,\F')$ and that in~\cite{MR2202309}. Assume $B$ is a unital $G$-C$^*$-algebra such that $B^G=\C1$. We can define a weak unitary tensor functor $\E\colon\Rep G\to\Hilb_f$, called the spectral functor, by letting
$$
\E(U)=\D_B(B,B\times U)
\ \ \text{and}\ \
\E_{2;U,V}\colon\E(U)\otimes\E(V)\to\E(U\circt V), \ \ T\otimes S\mapsto (T\otimes\iota)S.
$$
The scalar product on $\E(U)$ is defined by $S^*T=(T,S)1$, which makes sense by the ergodicity assumption. In general the maps $\E_{2;U,V}$ are not unitary but only isometric. When they are unitary, so that $(\E,\E_{2})$ becomes a unitary tensor functor, then the ergodic action of $G$ on $B$ is said to be of full quantum multiplicity. In this case, if~$G$~is reduced, then $B\cong B(\F,\E)$ as $G$-C$^*$-algebras~\cite{MR2202309} (see also~\cite{neshveyev-mjm-categorification}, where a more general result is proved in the notation consistent with the present work). In particular, another way of formulating the Hopf--Galois condition is to say that the action of $G$ is of full quantum multiplicity.

The spectral functor is constructed in a simple way using only the action of $G$, while in order to construct a tensor functor in our approach we also have to use the Miyashita--Ulbrich action. The reason why the two constructions give isomorphic functors is basically the following observation. Given a unitary fiber functor $\F'\colon\Rep G\to\Hilb_f$, we can define a new unitary fiber functor $\E\colon\Rep G\to\Hilb_f$ by letting
$$
\E(U)=\Hom(\C,\F'(U))\ \ \text{and}\ \ \E_{2;U,V}\colon\E(U)\otimes \E(V)\to\E(U\circt V),\ \
T\otimes S\mapsto\F'_{2;U,V}(T\otimes\iota)S.
$$
But it is clear that under the identification of $\Hom(\C,H)$ with $H$, the tensor functor $\E$ becomes identical to~$\F'$.

\bigskip

\section{Noncommutative Poisson boundaries}
\label{sec:poisson-boundary}

In this section we show that Theorem~\ref{tcatch} provides a link between Izumi's theory of Poisson boundaries of discrete quantum groups~\cite{MR1916370} and categorical Poisson boundaries introduced in~\cite{NY-categorical-Poisson-boundary}. We start by giving a categorical description of discrete duals.

\subsection{Discrete dual}
\label{sec:discrete-dual}

Consider the algebra $\ell^\infty(\hat G)\subset\U(G)=\C[G]^*$ of bounded functions on $\hat G$. We have a left adjoint action $\alpha$ of $G$ on $$\ell^\infty(\hat G)\cong\ell^\infty\mhyph\bigoplus_s B(H_s)$$ defined by
\begin{equation}\label{eadjoint}
B(H_s)\ni T\mapsto (U_s)_{21}^*(1\otimes T)(U_s)_{21}.
\end{equation}
This action is continuous only in the von Neumann algebraic sense, so in order to stay within the class of $G$-C$^*$-algebras, instead of $\ell^\infty(\hat G)$ we should rather consider the norm closure $B(\hat G)$ of the regular subalgebra $\ell^\infty_{\alg}(\hat G)\subset\ell^\infty(\hat G)$. Then the right action $\Dhat$ of $\hat G$ on $\ell^\infty(\hat G)$ makes this algebra into a unital braided-commutative Yetter--Drinfeld C$^*$-algebra. In other words, the left $\C[G]$-module structure on~$\ell^\infty_{\alg}(\hat G)$ is defined by
\begin{equation}\label{ead}
x\rhd a=(\iota\otimes x)\Dhat(a).
\end{equation}
In the subsequent computations we will use the notation $\Dhat(a)=a^{(1)}\otimes a^{(2)}$. Literally this does not make sense, but the expressions like $a^{(1)}\otimes \pi_U(a^{(2)})$ are still meaningful, since $(\iota\otimes \pi_U)\Dhat(a)$ is an element of the algebraic tensor product $\ell^\infty(\hat G)\otimes B(H_U)$.

\smallskip

We want to describe the corresponding C$^*$-tensor category $\CC=\CC_{B(\hat G)}$ and the unitary tensor functor $\F=\F_{B(\hat G)}\colon\Rep G\to\CC$. By definition, the category $\CC$ is the idempotent completion of the category with the same objects as in $\Rep G$, but with the morphism sets
$\CC(U,V)\subset B(H_U,H_V)\otimes\ell^\infty_{\alg}(\hat G)$. In fact, for the reasons that will become apparent in a moment, it is more convenient to consider $\CC(U,V)$ as a subset of $\ell^\infty_\alg(\hat G)\otimes B(H_U,H_V)$.
Thus, we define $\CC(U,V)$ as the set of elements $T\in \ell^\infty_\alg(\hat G)\otimes B(H_U,H_V)$ such that
$$
V^*_{31}(\alpha\otimes\iota)(T)U_{31}=1\otimes T.
$$
From the definition of the adjoint action~$\alpha$ we see that an element $T\in\ell^\infty_\alg(\hat G)\otimes B(H_U,H_V)$ lies in $\CC(U,V)$ if and only if it defines a $G$-equivariant map $H_s\otimes H_U\to H_s\otimes H_V$ for all~$s$. It follows that $\CC(U,V)$ can be identified with the space $\Nat_b(\iota\otimes U,\iota\otimes V)$ of bounded natural transformations between the functors $\iota\otimes U$ and $\iota\otimes V$ on $\Rep G$.

Using this picture we get a natural tensor structure on $\CC$: the tensor product of objects is defined as in $\Rep G$, while the tensor product of natural transformations $\nu\colon \iota\otimes U\to\iota\otimes V$ and $\eta\colon\iota\otimes W\to\iota\otimes Z$ is defined by
$$
\nu\otimes\eta=(\nu\otimes\iota_Z)(\iota_U\otimes\eta) =(\iota_V\otimes\eta)(\nu\otimes\iota_W),
$$
where $\nu\otimes\iota_Z$ is defined by $(\nu\otimes\iota_Z)_X=\nu_X\otimes\iota_Z$, while $\iota_U\otimes \eta$ is defined by $(\iota_U\otimes\eta)_X=\eta_{X\circt U}$. Explicitly, if $\nu=\sum_ia_i\otimes T_i\in \ell^\infty_{\alg}(\hat G)\otimes B(H_U,H_V)$ and $\eta=\sum_jb_j\otimes S_j\in \ell^\infty_{\alg}(\hat G)\otimes B(H_W,H_Z)$, then
\begin{equation} \label{etensord}
\nu\otimes\eta=\sum_{i,j}a_ib_j^{(1)}\otimes (T_i\pi_U(b_j^{(2)})\otimes S_j)\in\ell^\infty_{\alg}(\hat G)\otimes B(H_U\otimes H_W,H_V\otimes H_Z).
\end{equation}

The functor $\F\colon\Rep G\to\CC$ is now the strict tensor functor such that $\F(U)=U$ on objects and $\F(T)=1\otimes T$ on morphisms.

It remains to show that the tensor structure on $\CC$ defines the same $\C[G]$-module structure on~$\ell^\infty_{\alg}(\hat G)$ as \eqref{ead}. Consider an element $\bar\zeta\otimes T\in\bar H_V\otimes\CC(\un,V)$. Identifying $B(\C,H_V)$ with $H_V$ we can write $T=\sum_ka_k\otimes \zeta_k$ for some $a_k\in\ell^\infty_{\alg}(\hat G)$ and $\zeta_k\in H_V$. Then, identifying the algebra~$\B$ constructed from the pair $(\CC,\F)$ with $\ell^\infty_\alg(\hat G)$, the element $a=\pi(\bar\zeta\otimes T)\in\B=\ell^\infty_{\alg}(\hat G)$ equals $\sum_k(\zeta_k,\zeta)a_k$, see equation~\eqref{eiso}. Choose a unitary representations $U$ and an orthonormal basis~$\{\xi_i\}_i$ in~$H_U$ consisting of eigenvectors of $\rho$, so $\rho\xi_i=\rho_i\xi_i$. By \eqref{etensord} the morphism
$$
(\iota\otimes T\otimes\iota)\bar R_U\in\CC(\un,U\circt V\circt \bar U)\subset\ell^\infty(\hat G)\otimes(H_U\otimes H_V\otimes \bar H_U)
$$
is represented by the element
$$
\sum_{k,l}a_k^{(1)} \otimes\big(\rho_l^{1/2}a_k^{(2)}\xi_l\otimes\zeta_k\otimes\bar\xi_l\big).
$$
Then by definition~\eqref{eq:from-tensor-to-dual-action} of the map~$\tilde\rhd$ we get
$$
(\bar\xi_i\otimes\xi_j)\trhd\left(\bar\zeta\otimes\left(\sum_ka_k\otimes\zeta_k\right)\right)
=\overline{(\xi_i\otimes\zeta\otimes\rho_j^{-1/2}\bar\xi_j)}\otimes\left(\sum_{k,l}a_k^{(1)} \otimes\big(\rho_l^{1/2}a_k^{(2)}\xi_l\otimes\zeta_k\otimes\bar\xi_l\big)\right),
$$
whence
$$
u_{ij}\rhd a=\sum_k(a_k^{(2)}\xi_j,\xi_i)(\zeta_k,\zeta)a_k^{(1)} =(a^{(2)}\xi_j,\xi_i)a^{(1)}.
$$
But this is exactly how the action \eqref{ead} is defined.

\subsection{Poisson boundaries}
\label{sec:categ-appr-poiss}

Let us briefly overview the theory of noncommutative Poisson boundaries developed by Izumi~\cite{MR1916370}.

For a finite dimensional unitary representation $U$ of $G$, consider the state $\phi_U$ on~$B(H_U)$ defined~by
\begin{equation} \label{ephi}
\phi_U(T) = \frac{\Tr(T \pi_U(\rho)^{-1})}{\dim_q U} \quad \text{for}\ \ T \in B(H).
\end{equation}
If $U$ is irreducible, it can be characterized as the unique state satisfying
$$
(\iota \otimes \phi_U)(U_{2 1}^* (1 \otimes T) U_{2 1}) = \phi_U(T).
$$
For our fixed representatives of irreducible representations $\{U_s\}_s$ of $G$, we write $\phi_s$ instead of $\phi_{U_s}$.

When $\phi$ is a normal state on $\ell^\infty(\hat{G})$, we define a completely positive map $P_\phi$ on $\ell^\infty(\hat{G})$ by
$$
P_{\phi}(a) = (\phi \otimes \iota) \Dhat(a).
$$
If $\mu$ is a probability measure on the set $\Irr(G)$ of isomorphism classes of irreducible representations of $G$, we define a normal unital completely positive map $P_\mu$ on $\ell^\infty(\hat G)$ by $P_\mu = \sum_s \mu(s) P_{\phi_s}$.  The space
$$
H^\infty(\hat{G}, \mu) = \{ x \in \ell^\infty(\hat{G}) \mid x = P_\mu(x) \}
$$
of $P_\mu$-harmonic elements is called the \emph{noncommutative Poisson boundary} of $\hat{G}$ with respect to $\mu$.  This is an operator subspace of $\ell^\infty(\hat{G})$ closed under the left adjoint action $\alpha$ of $G$ defined by \eqref{eadjoint} and the right action $\Dhat$ of $\hat G$ on itself by translations.  It has a new product structure
$$
x \cdot y = \lim_{n \to \infty} P_\mu^n(x y),
$$
where the limit is taken in the strong$^*$ operator topology.  With this product $H^\infty(\hat{G}, \mu)$ becomes a von Neumann algebra (with the original operator space structure), and the actions of~$G$ and~$\hat{G}$ on~$\ell^\infty(\hat G)$ define continuous, in the von Neumann algebraic sense, actions on $H^\infty(\hat G,\mu)$.

Consider the regular subalgebra $H^\infty_\alg(\hat{G}, \mu) = H^\infty(\hat{G}, \mu) \cap \ell^\infty_\alg(\hat{G})$ of $H^\infty(\hat{G}, \mu)$ and denote by $B(\hat G,\mu)$ its norm closure. In other words, in the notation of Section~\ref{sec:discrete-dual}, $B(\hat G,\mu)=B(\hat G)\cap H^\infty(\hat G,\mu)$. We will show in Theorem~\ref{thm:tensor-cat-compar-Pois-bdry} that the action of~$\hat G$ on $H^\infty(\hat{G}, \mu)$ restricts to a continuous action on~$B(\hat G,\mu)$ and that $B(\hat G,\mu)$ becomes a braided-commutative Yetter--Drinfeld $G$-C$^*$-algebra.

\medskip

Let us now recall the construction of the Poisson boundary of $(\Rep G,\mu)$ defined in~\cite{NY-categorical-Poisson-boundary}.

The image of $\Hom_G(U\circt V,U\circt W)$ under the map
$$
\phi_U\otimes\iota\colon B(H_U)\otimes B(H_V,H_W)\to B(H_V,H_W)
$$
is contained in $\Hom_G(V,W)$, and the maps
$$
\phi_U\otimes\iota\colon \Hom_G(U\circt V,U\circt W)\to \Hom_G(V,W)
$$
we thus get, are what we called the partial categorical traces on $\Rep G$ in~\cite{NY-categorical-Poisson-boundary}. They allow us to define an operator $P_U$ on the space of natural transformations $\Nat(\iota \otimes V, \iota \otimes W)$ by
$$
P_U(\eta)_X = (\phi_U \otimes \iota)(\eta_{U\otimes X}).
$$
It is easy to see that this operation preserves the subspace $\Nat_b(\iota \otimes V, \iota \otimes W)$ of bounded natural transformations. Given a probability measure $\mu$ on $\Irr(G)$, we define an operator $P_\mu$ acting on $\Nat_b(\iota \otimes V, \iota \otimes W)$ by $P_\mu=\sum_s \mu(s)P_{U_s}$.

A bounded natural transformation $\eta\colon\iota\otimes V\to\iota\otimes W$ is called $P_\mu$-{\em harmonic} if $P_\mu(\eta)=\eta$. Any morphism $T\colon V\to W$ defines a bounded natural transformation $(\iota_X\otimes T)_X$, which is obviously $P_\mu$-harmonic for every $\mu$.

The \emph{categorical Poisson boundary} $(\PP,\Pi)$ of $(\Rep G,\mu)$ consists of the C$^*$-tensor category~$\PP$ and the strict unitary tensor functor $\Pi\colon\CC\to\PP$ defined as follows. The category $\PP$ is the idempotent completion of $\Rep G$ with the new morphism sets
$$
\PP(U, V) = \{ \eta \in \Nat_b(\iota \otimes U, \iota \otimes V) \mid P_\mu(\eta) = \eta \},
$$
endowed with the composition law
$$
(\eta \cdot \nu)_X = \lim_{n \to \infty}P^n_\mu(\eta \nu)_X.
$$
On objects in $\Rep G$ the tensor product in $\PP$ is the same as in $\Rep G$, while on morphisms it is given~by
$$
\eta\otimes\nu=(\eta\otimes\iota)\cdot(\iota\otimes \nu) =(\iota\otimes \nu)\cdot(\eta\otimes\iota),
$$
where $\eta\otimes\iota$ and $\iota\otimes \nu$ are defined as in Section~\ref{sec:discrete-dual}. The functor $\Pi\colon\CC\to\PP$ is defined by letting $\Pi(U)=U$ on objects and $\Pi(T)=(\iota_X\otimes T)_X$ on morphisms. We usually omit $\Pi$ and consider $\Rep G$ as a subcategory of $\PP$.

\begin{theorem}
\label{thm:tensor-cat-compar-Pois-bdry}
Let $G$ be a compact quantum group and $\mu$ be a probability measure on $\Irr(G)$. Then the dense C$^*$-subalgebra $B(\hat G,\mu)\subset H^\infty(\hat G,\mu)$ is a unital braided-commutative Yetter--Drin\-feld $G$-C$^*$-algebra and the pair $(\D_{B(\hat G,\mu)},\E_{B(\hat G,\mu)})$, consisting of the C$^*$-tensor category~$\D_{B(\hat G,\mu)}$ of $G$-equivariant finitely generated Hilbert $B(\hat G,\mu)$-modules and the unitary tensor functor~$\E_{B(\hat G,\mu)}\colon\Rep G\to\D_{B(\hat G,\mu)}$, is isomorphic to the categorical Poisson boundary of $(\Rep G,\mu)$.
\end{theorem}

\begin{proof} When $\mu=\delta_e$, in which case $H^\infty(\hat G,\mu)=\ell^\infty(\hat G)$, this theorem is the contents of Section~\ref{sec:discrete-dual}. The general case easily follows from this. Indeed, denote by $\tilde\B_\mu$ and $\B_\mu$ the algebras constructed from the Poisson boundary $(\PP,\Pi)$ of $(\Rep G,\mu)$ as described in Section~\ref{sec:from-tens-categ}. If $\mu=\delta_e$, we simply write $\tilde\B$ and $\B$. Thus,
$$
\tilde\B=\bigoplus_U(\bar H_U\otimes\Nat_b(\iota,\iota\otimes U)).
$$
As we showed in Section~\ref{sec:discrete-dual}, the Yetter--Drinfeld algebra $\B$ can be identified with $\ell^\infty_\alg(\hat G)$, and then the homomorphism $\pi\colon\tilde\B\to\B=\ell^\infty_\alg(\hat G)$ is given by
$$
\pi\left(\bar\xi\otimes\left(\sum_k a_k\otimes\zeta_k\right)\right)=\sum_k(\zeta_k,\xi)a_k,
$$
if we view $\Nat_b(\iota,\iota\otimes U)$ as a subspace of $\ell^\infty\text{-}\bigoplus_s(B(H_s)\otimes H_U)$.

The Markov operators $P_\mu$ on $\Nat_b(\iota,\iota\otimes U)$ define an operator $\iota\otimes P_\mu$ on $\tilde\B$. Then by definition, the algebra $\tilde\B_\mu$ is the subspace of $(\iota\otimes P_\mu)$-invariant elements in $\tilde\B$. Furthermore, by construction we have $\pi(\iota\otimes P_\mu)=P_\mu\pi$, where on the right hand side by $P_\mu$ we mean the operator on~$\ell^\infty(\hat G)$ used to define the Poisson boundary of $\hat G$. This already implies that the restriction of $\pi$ to $\tilde\B_\mu$ defines a surjective homomorphism $\tilde\B_\mu\to H^\infty_\alg(\hat G,\mu)$. Recalling how $\B_\mu$ is obtained from $\tilde\B_\mu$, we then conclude that this restriction factors through $\B_\mu$ and defines a $G$-equivariant $*$-isomorphism $\B_\mu\cong H^\infty_\alg(\hat G,\mu)$.

It remains to compare the $\C[G]$-module structures. For this part the computation is in fact exactly the same as for $\mu=\delta_e$.  The point is that, in the formula~\eqref{eq:from-tensor-to-dual-action} for the $\C[G]$-action, one only needs to compute the compositions of the form $(\iota \otimes T \otimes \iota)\bar{R}_U$ for $U, V \in \Rep G$ and $T \in \PP(\un, V)$. In general, if $\eta \in \PP(U, V)$ and $S \in W\to U$ is a morphism in $\Rep G$, the composition $\eta \cdot S$ is represented by the family $(\eta_X (\iota_X \otimes S))_X$, which is independent of $\mu$.  Thus, the $\C[G]$-module structure on $H^\infty_\alg(\hat{G}, \mu)$ induced by the tensor category structure of $\PP$ via the isomorphism $\B_\mu\cong H^\infty_\alg(\hat G,\mu)$, is the restriction of that on $\ell^\infty_\alg(\hat{G})$.  But this is exactly how the original $\C[G]$-module structure was defined on $H^\infty_\alg(\hat{G}, \mu)$.
\end{proof}

Recall that a probability measure $\mu$ on $\Irr(G)$ is called \emph{ergodic}, if the only $P_\mu$-harmonic functions on $\Irr(G)$ are the constant functions, that is,  $H^\infty(\hat{G}, \mu)^G$ reduces to $\C1$. Such a measure exists if and only if $\Irr(G)$ is at most countable and $\Rep G$ is weakly amenable, see~\cite{NY-categorical-Poisson-boundary}*{Sections~2 and~7.1}. From the above theorem and our results on categorical Poisson boundaries in~\cite{NY-categorical-Poisson-boundary} we then get the following theorem, originally proved by Tomatsu~\cite{MR2335776}*{Theorem~4.8}. (To be more precise, Tomatsu formulates the result in a more restricted form, but his proof shows that a stronger result formulated below is true.)

\begin{theorem}
\label{thm:coamen-Poi-bdry-quot-max-Kac}
Let $G$ be a coamenable compact quantum group, and $\mu$ be an ergodic probability measure on $\Irr(G)$.  Then the Poisson boundary $H^\infty(\hat{G}, \mu)$ is $G$- and $\hat G$-equivariantly isomorphic to $L^\infty(G/K)$, where $K$ is the maximal Kac quantum subgroup of $G$.
\end{theorem}

\bp The results of~\cite{NY-categorical-Poisson-boundary}*{Section~4} imply that the Poisson boundary of $(\Rep G,\mu)$ is isomorphic to the forgetful functor $\F\colon\Rep G\to\Rep K$, see~\cite{classification}*{Section~2} for details. From Theorem~\ref{thm:tensor-cat-compar-Pois-bdry} and the discussion in Section~\ref{sec:quot-type-coid}, where we showed that we have an isomorphism $(\D_{C(G/K)},\E_{C(G/K)})\cong(\Rep K,\F)$, we conclude that there exists a $G$- and $\hat G$-equivariant isomorphism $B(\hat G,\mu)\cong C(G/K)$. Since $H^\infty(\hat G,\mu)$ and $L^\infty(G/K)$ are the von Neumann algebras generated by $B(\hat G,\mu)$ and $C(G/K)$, respectively, in the GNS-representations defined by the unique $G$-invariant states, we conclude that $H^\infty(\hat G,\mu)\cong L^\infty(G/K)$.
\ep

Of course, conversely, the argument of Tomatsu could be used to show
that the Poisson boundary of $(\Rep G,\mu)$ is
$\F\colon\Rep G\to\Rep K$ without relying
on~\cite{NY-categorical-Poisson-boundary}*{Section~4}.

Note that in order to prove Theorem~\ref{thm:coamen-Poi-bdry-quot-max-Kac} we do not need the full strength of Theorem~\ref{tcatch}, it suffices to understand how $B(\hat G,\mu)$ and $C(G/K)$ are reconstructed from the functors $\Pi\colon\Rep G\to\PP$ and $\F\colon\Rep G\to\Rep K$. It is also worth noting that independently of which approach to  Theorem~\ref{thm:coamen-Poi-bdry-quot-max-Kac} one prefers, all the results of this type have so far relied in a crucial, but every time different, way on the so called Izumi's Poisson integral~\citelist{\cite{MR1916370}\cite{MR2200270}\cite{MR2335776} \cite{NY-categorical-Poisson-boundary}}.

\smallskip

We finish the paper by proving a converse to
Theorem~\ref{thm:coamen-Poi-bdry-quot-max-Kac}.

\begin{proposition}
Let $G$ be a compact quantum group and $\mu$ be a probability measure on $\Irr(G)$. Assume that the Poisson boundary $H^\infty(\hat G,\mu)$ is $G$- and $\hat G$-equivariantly isomorphic to $L^\infty(G/H)$ for a closed quantum subgroup $H$ of $G$. Then $G$ is coamenable, and hence $H$ is the maximal Kac quantum subgroup of~$G$.
\end{proposition}

\bp Theorem~\ref{thm:tensor-cat-compar-Pois-bdry} and the assumptions of the proposition imply that the Poisson boundary of $(\Rep G,\mu)$ is isomorphic to the forgetful functor $\F\colon\Rep G\to\Rep H$. On the other hand, since the action of $G$ on $L^\infty(G/H)$ is ergodic, the measure $\mu$ is ergodic, and therefore by~\cite{NY-categorical-Poisson-boundary}*{Theorem~5.1} the Poisson boundary of $(\Rep G,\mu)$ defines the amenable dimension function on $\Rep G$. It follows that the classical dimension function on $\Rep G$ is amenable, which exactly means that $G$ is coamenable.
\ep

\bigskip


\bigskip


\begin{bibdiv}
\begin{biblist}

\bib{arXiv:1006.3014}{misc}{
      author={Bichon, Julien},
       title={Hopf--Galois objects and cogroupoids},
         how={lecture notes},
        date={2010},
      eprint={\href{http://arxiv.org/abs/1006.3014}{{\tt arXiv:1006.3014
  [math.QA]}}},
}

\bib{MR2202309}{article}{
      author={Bichon, Julien},
      author={De~Rijdt, An},
      author={Vaes, Stefaan},
       title={Ergodic coactions with large multiplicity and monoidal
  equivalence of quantum groups},
        date={2006},
        ISSN={0010-3616},
     journal={Comm. Math. Phys.},
      volume={262},
      number={3},
       pages={703\ndash 728},
         url={http://dx.doi.org/10.1007/s00220-005-1442-2},
         doi={10.1007/s00220-005-1442-2},
      eprint={\href{http://arxiv.org/abs/math/0502018}{{\tt arXiv:math/0502018 
  [math.OA]}}},
      review={\MR{2202309 (2007a:46072)}},
}

\bib{MR2863377}{article}{
      author={Brugui{{\`e}}res, Alain},
      author={Natale, Sonia},
       title={Exact sequences of tensor categories},
        date={2011},
        ISSN={1073-7928},
     journal={Int. Math. Res. Not. IMRN},
      number={24},
       pages={5644\ndash 5705},
         url={http://dx.doi.org/10.1093/imrn/rnq294},
         doi={10.1093/imrn/rnq294},
      eprint={\href{http://arxiv.org/abs/1006.0569}{{\tt arXiv:1006.0569 
  [math.QA]}}}, 
      review={\MR{2863377}},
}

\bib{MR1291020}{inproceedings}{
      author={Caenepeel, S.},
      author={Van~Oystaeyen, F.},
      author={Zhang, Y.~H.},
       title={Quantum {Y}ang-{B}axter module algebras},
        date={1994},
   booktitle={Proceedings of {C}onference on {A}lgebraic {G}eometry and {R}ing
  {T}heory in honor of {M}ichael {A}rtin, {P}art {III} ({A}ntwerp, 1992)},
      volume={8},
       pages={231\ndash 255},
         url={http://dx.doi.org/10.1007/BF00960863},
         doi={10.1007/BF00960863},
      review={\MR{1291020 (95e:16031)}},
}

\bib{MR1289088}{article}{
      author={Cohen, Miriam},
      author={Westreich, Sara},
       title={From supersymmetry to quantum commutativity},
        date={1994},
        ISSN={0021-8693},
     journal={J. Algebra},
      volume={168},
      number={1},
       pages={1\ndash 27},
         url={http://dx.doi.org/10.1006/jabr.1994.1217},
         doi={10.1006/jabr.1994.1217},
      review={\MR{1289088 (95e:16033)}},
}

\bib{MR3039775}{article}{
      author={Davydov, Alexei},
      author={M{{\"u}}ger, Michael},
      author={Nikshych, Dmitri},
      author={Ostrik, Victor},
       title={The {W}itt group of non-degenerate braided fusion categories},
        date={2013},
        ISSN={0075-4102},
     journal={J. Reine Angew. Math.},
      volume={677},
       pages={135\ndash 177},
      eprint={\href{http://arxiv.org/abs/1009.2117}{{\tt arXiv:1009.2117
  [math.QA]}}},
      review={\MR{3039775}},
}

\bib{MR3121622}{article}{
      author={De~Commer, Kenny},
      author={Yamashita, Makoto},
       title={Tannaka-{K}re\u\i n duality for compact quantum homogeneous
  spaces. {I}. {G}eneral theory},
        date={2013},
        ISSN={1201-561X},
     journal={Theory Appl. Categ.},
      volume={28},
       pages={No. 31, 1099\ndash 1138},
      eprint={\href{http://arxiv.org/abs/1211.6552}{{\tt arXiv:1211.6552
  [math.OA]}}},
      review={\MR{3121622}},
}

\bib{MR1916370}{article}{
      author={Izumi, Masaki},
       title={Non-commutative {P}oisson boundaries and compact quantum group
  actions},
        date={2002},
        ISSN={0001-8708},
     journal={Adv. Math.},
      volume={169},
      number={1},
       pages={1\ndash 57},
         url={http://dx.doi.org/10.1006/aima.2001.2053},
         doi={10.1006/aima.2001.2053},
      review={\MR{1916370 (2003j:46105)}},
}

\bib{MR2200270}{article}{
      author={Izumi, Masaki},
      author={Neshveyev, Sergey},
      author={Tuset, Lars},
       title={Poisson boundary of the dual of {${\rm SU}_q(n)$}},
        date={2006},
        ISSN={0010-3616},
     journal={Comm. Math. Phys.},
      volume={262},
      number={2},
       pages={505\ndash 531},
         url={http://dx.doi.org/10.1007/s00220-005-1439-x},
         doi={10.1007/s00220-005-1439-x},
      eprint={\href{http://arxiv.org/abs/math/0402074}{{\tt arXiv:math/0402074
  [math.OA]}}},
      review={\MR{MR2200270 (2007f:58012)}},
}

\bib{MR1190512}{incollection}{
      author={Landstad, Magnus~B.},
       title={Ergodic actions of nonabelian compact groups},
        date={1992},
   booktitle={Ideas and methods in mathematical analysis, stochastics, and
  applications ({O}slo, 1988)},
   publisher={Cambridge Univ. Press, Cambridge},
       pages={365\ndash 388},
      review={\MR{1190512 (93j:46072)}},
}

\bib{neshveyev-mjm-categorification}{article}{
      author={Neshveyev, Sergey},
       title={Duality theory for nonergodic actions},
     journal={M\"unster J. Math.},
      eprint={\href{http://arxiv.org/abs/1303.6207}{{\tt arXiv:1303.6207
  [math.OA]}}},
        note={to appear in print},
}

\bib{neshveyev-tuset-book}{book}{
      author={Neshveyev, Sergey},
      author={Tuset, Lars},
      title={Compact quantum groups and their representation categories},
      series={Cours Sp\'ecialis\'es [Specialized Courses]},
      volume={20},
      publisher={Soci\'et\'e Math\'ematique de France, Paris},
      date={2013},
      pages={168},
      isbn={978-2-85629-777-3},
      note={preliminary version available at \url{http://folk.uio.no/sergeyn/papers/CQGRC.pdf}},
}

\bib{NY-categorical-Poisson-boundary}{misc}{
      author={Neshveyev, Sergey},
      author={Yamashita, Makoto},
      title={Poisson boundaries of monoidal categories},
      date={2014},
      how={preprint},
      eprint={\href{http://arxiv.org/abs/1405.6572}{{\tt arXiv:1405.6572[math.OA]}}},
}

\bib{classification}{misc}{
      author={Neshveyev, Sergey},
      author={Yamashita, Makoto},
       title={Classification of non-{K}ac compact quantum groups of $\mathrm{SU}(n)$ type},
       how={preprint},
       date={2014},
       eprint={\href{http://arxiv.org/abs/1405.6574}{{\tt arXiv:1405.6574[math.QA]}}},
}

\bib{MR2566309}{article}{
      author={Nest, Ryszard},
      author={Voigt, Christian},
       title={Equivariant {P}oincar\'e duality for quantum group actions},
        date={2010},
        ISSN={0022-1236},
     journal={J. Funct. Anal.},
      volume={258},
      number={5},
       pages={1466\ndash 1503},
      eprint={\href{http://arxiv.org/abs/0902.3987}{{\tt arXiv:0902.3987
  [math.KT]}}},
         url={http://dx.doi.org/10.1016/j.jfa.2009.10.015},
         doi={10.1016/j.jfa.2009.10.015},
      review={\MR{2566309}},
}

\bib{MR2785890}{article}{
      author={Salmi, Pekka},
       title={Compact quantum subgroups and left invariant {$C^*$}-subalgebras
  of locally compact quantum groups},
        date={2011},
        ISSN={0022-1236},
     journal={J. Funct. Anal.},
      volume={261},
      number={1},
       pages={1\ndash 24},
         url={http://dx.doi.org/10.1016/j.jfa.2011.03.003},
         doi={10.1016/j.jfa.2011.03.003},
      eprint={\href{http://arxiv.org/abs/1004.4161}{{\tt arXiv:1004.4161  [math.OA]}}},
      review={\MR{2785890 (2012f:46159)}},
}

\bib{MR1408508}{article}{
      author={Schauenburg, Peter},
       title={Hopf bi-{G}alois extensions},
        date={1996},
        ISSN={0092-7872},
     journal={Comm. Algebra},
      volume={24},
      number={12},
       pages={3797\ndash 3825},
         url={http://dx.doi.org/10.1080/00927879608825788},
         doi={10.1080/00927879608825788},
      review={\MR{1408508 (97f:16064)}},
}

\bib{MR2335776}{article}{
      author={Tomatsu, Reiji},
       title={A characterization of right coideals of quotient type and its
  application to classification of {P}oisson boundaries},
        date={2007},
        ISSN={0010-3616},
     journal={Comm. Math. Phys.},
      volume={275},
      number={1},
       pages={271\ndash 296},
         url={http://dx.doi.org/10.1007/s00220-007-0267-6},
         doi={10.1007/s00220-007-0267-6},
      eprint={\href{http://arxiv.org/abs/math/0611327}{{\tt arXiv:math/0611327
  [math.OA]}}},
      review={\MR{2335776 (2008j:46058)}},
}

\bib{MR1006625}{article}{
      author={Ulbrich, K.-H.},
       title={Fibre functors of finite-dimensional comodules},
        date={1989},
        ISSN={0025-2611},
     journal={Manuscripta Math.},
      volume={65},
      number={1},
       pages={39\ndash 46},
         url={http://dx.doi.org/10.1007/BF01168365},
         doi={10.1007/BF01168365},
      review={\MR{1006625 (90e:16014)}},
}

\bib{MR990110}{article}{
      author={Wassermann, Antony},
       title={Ergodic actions of compact groups on operator algebras. {II}.
  {C}lassification of full multiplicity ergodic actions},
        date={1988},
        ISSN={0008-414X},
     journal={Canad. J. Math.},
      volume={40},
      number={6},
       pages={1482\ndash 1527},
      review={\MR{MR990110 (92d:46168)}},
}

\end{biblist}
\end{bibdiv}
\end{document}